\documentclass[journal,twoside,web]{ieeecolor}
\usepackage{generic}
\usepackage{cite}
\usepackage{amsmath,amssymb,amsfonts}
\usepackage{algorithmic}
\usepackage{graphicx}
\usepackage{textcomp}
\usepackage{braket}
\usepackage{mathabx}
\usepackage[english]{babel}

\usepackage{amsthm}
\usepackage{bbold}
\usepackage{pgfplots}
\usepackage[normalem]{ulem}
\usepackage[font={small}]{caption}
\usepackage{subfig}
\usepackage{url}
\usepackage{xcolor}

\let\labelindent\relax

\usepackage{enumitem} 

\definecolor{OliveGreen}{rgb}{0, 0.5, 0}

\usepackage[lined,linesnumbered,algoruled]{algorithm2e}

\newcommand{\ifinclude}[1]{#1}
\renewcommand{\ifinclude}[1]{ }

\usepackage{accents}

\renewcommand{\baselinestretch}{.995}

\newtheorem{theorem}{Theorem}[section]
\newtheorem{definition}[theorem]{Definition}

\newtheorem{lemma}[theorem]{Lemma}

\newtheorem{remark}[theorem]{Remark}

\newtheorem{example}[theorem]{Example}
\newtheorem{corollary}[theorem]{Corollary}
\newtheorem{proposition}[theorem]{Proposition}

\allowdisplaybreaks

\newcommand{\CC}{\mathcal{C}}
\newcommand{\EE}{\mathcal{E}}
\newcommand{\FF}{\mathcal{F}}
\newcommand{\GG}{\mathcal{G}}
\newcommand{\HH}{\mathcal{H}}
\newcommand{\II}{\mathcal{I}}
\newcommand{\KK}{\mathcal{K}}
\newcommand{\PP}{\mathcal{P}}
\newcommand{\QQ}{\mathcal{Q}}
\newcommand{\RR}{\mathcal{R}}
\newcommand{\VV}{\mathcal{V}}
\newcommand{\WW}{\mathcal{W}}

\newcommand{\real}{{\mathbb{R}}}

\newcommand{\realnonnegative}{{\mathbb{R}}_{\ge 0}}
\newcommand{\naturalnumbers}{\mathbb{N}}
\newcommand{\setdef}[2]{\{#1 \; | \; #2\}}
\newcommand{\setdefb}[2]{\bigl\{#1 \; | \; #2\bigr\}}
\newcommand{\setdefB}[2]{\Bigl\{#1 \; | \; #2\Bigr\}}
\newcommand{\map}[3]{#1:#2 \rightarrow #3}
\newcommand{\abs}[1]{\ensuremath{\left\lvert{#1}\right\rvert}}
\newcommand{\norm}[1]{\ensuremath{\| #1 \|}}

\newcommand{\info}{\mathrm{inf}}

\newcommand{\lowra}{\mathtt{low}}
\newcommand{\upra}{\mathtt{up}}
\newcommand{\tr}{\mathrm{true}}
\newcommand{\stage}{\mathrm{stage}}

\newcommand{\vin}{v^{\mathsf{in}}}
\newcommand{\vout}{v^{\mathsf{out}}}
\newcommand{\VI}{\operatorname{VI}}

\newcommand{\oprocendsymbol}{\hbox{$\bullet$}}
\newcommand{\oprocend}{\relax\ifmmode\else\unskip\hfill\fi\oprocendsymbol}

\newcommand{\colsto}[2]{\mathsf{CS}(#1,#2)}
\newcommand{\longthmtitle}[1]{\mbox{}\textup{\textsl{(#1):}}}

\usepackage[prependcaption,colorinlistoftodos]{todonotes}

\title{Inferring the prior in routing games using \\ public signalling}
\author{Jasper Verbree, \IEEEmembership{Student Member, IEEE}, Ashish Cherukuri \IEEEmembership{Member, IEEE}
	\thanks{J. Verbree and A. Cherukuri 
				are with the ENgineering and TEchnology institute Groningen (ENTEG), University of Groningen, 
				\texttt{\{j.verbree, a.k.cherukuri\}@rug.nl}.}
}

\begin{document}	
	\maketitle
	\thispagestyle{empty}
\begin{abstract}
This paper considers Bayesian persuasion for routing games where information about the uncertain state of the network is provided by
a traffic information system (TIS) using public signals. In this setup, the TIS commits to a signalling scheme and participants form a posterior belief about the state of the network based on prior beliefs and the received signal. They subsequently select routes minimizing their individual expected travel time under their posterior beliefs, giving rise to a Wardrop equilibrium. We investigate how the TIS can infer the prior beliefs held by the participants by designing suitable signalling schemes, and observing the equilibrium flows under different signals.  
We show that under mild conditions a signalling scheme that allows for exact inference of the prior exists. We then provide an iterative algorithm that finds such a scheme in a finite number of steps. We show that schemes designed by our algorithm are robust, in the sense that they can still identify the prior after a small enough perturbation. We also investigate the case where the population is divided among multiple priors, and give conditions under which the fraction associated to each prior can be identified. Several examples illustrate our results.
\end{abstract}
\section{Introduction}
\IEEEPARstart{R}{ecent} years have seen increased utilization of \textit{traffic information systems} (TISs) such as Google maps and Waze by users of traffic networks. While TISs can cause congestion and pose various challenges for traffic management~\cite{JM:19}, they also create the opportunity of \textit{information design}, where information about the state of the network is strategically revealed in order to minimize congestion. For instance, in some cases travel time of all participants can be reduced when information about certain routes is witheld~\cite{DA-AM-AM-AO:18}.
A fitting framework for studying the effects of information on decision making is \textit{Bayesian persuasion}~\cite{EK-MG:11}. Applying this framework to the routing context, the network is assumed to be in one of several possible states, each representing for instance the presence or absence of road congestion, accidents, or weather events.
The participants are assumed to have a prior belief  about the probability of each state occurring. The TIS releases information about the state using a set of messages or signals and in this way influences the posterior belief formed by the participants. Subsequently, participants select routes that minimize the expected travel cost in a selfish manner under the posterior belief, i.e., they route according to a \textit{Wardrop equilibrium}. In the above explained framework, the TIS can influence the flow by carefully designing the map from states to messages, also known as the signalling scheme. The effects of such a design naturally depend on the prior of the participants. However, the TIS may not know this prior in advance, presenting a problem for the implementation of this method. For instance, when aiming to minimize the total travel time of all participants, any error in the estimation of the prior by the TIS can result in decreased performance.

The aim of this paper is to address this problem by studying how the prior of a population influences the Wardrop equilibrium, and how information about the prior can be inferred from observing the equilibrium flows under a signalling scheme.
\subsubsection*{Literature review}
The Bayesian persuasion framework~\cite{EK-MG:11} for information design is adapted to the network routing setup in several recent works. In~\cite{SD-EK-RM:17} the potential of information design to reduce travel times is show-cased for two common examples, in~\cite{SV-MF-AH:15} the cost-performance of incentive-compatible signalling schemes are studied in comparison to socially optimal solutions, and in~\cite{OM-CL:18,HT-DT:19} the relative performance of different strategies of information design, such as public and private signalling, are obtained. Instead of assuming that all users participate in persuasion, the works~\cite{YZ-KS:18,YZ-KS:20} determine optimal information provision for heterogeneous populations, where a part of the users do not ``trust'' the TIS. Closer to the subject of this paper \cite{JL:19} also studies the effects of a mismatch between the actual distribution and the prior belief of a population concerning some parameters of a congestion game.  In particular it introduces a type of routing game called a `subjective Bayesian congestion game' which considers information that users have about the signals other users received. Recent works also investigate the possible pitfalls of information provision by TISs. For example,~\cite{MW-SA-AEO:21,HT-AS-KP-PV:19} explore inefficiencies caused by competing TISs;~\cite{DA-AM-AM-AO:18} highlights how knowing more routes can cause more congestion, revealing informational Braess's paradox; and \cite{GB-FP:20} demonstrates oscillating traffic behaviour when information about travel-times is available in real time. An analysis of how the benefits and decriments of revealing information to the population relate to the specifics of the cost functions and structure of the uncertainty is given in \cite{RL-TD-GE-AR:14}.

The viewpoint adapted in our paper of learning about private parameters, such as the prior, of users in a routing setup is similar in spirit to~\cite{KL-WK-AB:16} and~\cite{JT-AB:17}. In the former, the problem of estimating the learning rate of the population that employs a mirror descent algorithm to adapt route choices is considered. In the latter, learning of the cost functions of paths is studied. In a broader context, \cite{LJR-TF:20} investigates incentive design for a set of noncooperative agents by learning the cost functions that govern their decisions. Our work is partly related to learning in routing games, where a lot of focus is on learning from the perspective of participants, see~\cite{SK-WK-RD-AB:15,EM-FP-AO:17} and references therein. The work~\cite{MW-SA:19} looks at a Bayesian framework and explores how participants learn about the state of the network in repeated play. It is worth noting that none of the works consider learning preferences or biases inherently present in decisions of users in the context of information design. 

Finally, we note that a popular alternative to information design for influencing flows in a traffic network is \textit{incentive design}. For routing games, this area focuses on how tolls and subsidies can be used to influence the behavior of traffic participants, see \cite{PNB-JRM:17} and references therein for an overview, and \cite{BLF-PNB-JRM:22} for an investigation on the potential of using incentive and information design in tandem.
\subsubsection*{Setup and contributions}
We study non-atomic routing games over a network with a single origin and destination. The network can exist in one of a finite number of possible states and each path is associated with a cost function that differs per state of the network. Traffic participants rely on a traffic information system (TIS) to supply information about the current state of the system. The TIS commits in advance to a signalling scheme which is known to the participants and which determines the probability with which the TIS sends a particular signal to all participants when a state is realized. This framework is commonly known as \emph{public signalling} as everyone gets the same signal. After receiving a signal, users form a Bayesian posterior belief about the state of the network based on the prior belief and the signalling scheme employed by the TIS. Subsequently, the flow induced by user decisions is a Wardrop equilibrium with respect to the expected costs under the posterior. The aim of this paper is to investigate how and when the TIS can infer the prior exactly by observing the equilibrium flows under different signals. In Section \ref{sec:modelandproblem} we motivate the advantage of knowing the prior with an example showing that an error in the estimate of the prior can lead to an increase in social cost. Main contributions of this paper are:
\begin{enumerate}[label=(\roman*)]
	\item Using a constructive proof, we show that under mild conditions there always exists a signalling scheme employing as many signals as there are states that will allow the TIS to exactly determine the prior.
	\item We give an iterative procedure, terminating in a finite number of steps, that finds a signalling scheme allowing the TIS to determine the prior. The procedure uses the observations of equilibrium flows in each iteration.
	\item We show that a subclass of signalling schemes that allow for identification of the prior are robust; i.e., schemes of this class can identify the prior, even if the prior is subject to small perturbations between instances of the game.
\end{enumerate}
We provide examples throughout the paper to better illustrate the technical exposition.
\subsubsection*{Organization}
The routing model and the motivating example are presented in Section~\ref{sec:modelandproblem}. The existence and design of a signalling scheme that can infer the prior exactly are studied in Section~\ref{sec:general}. Some additional results considering relaxed assumptions on the prior are collected in Section~\ref{sec:relaxed-prior}. Finally, the conclusions are summarized in Section~\ref{sec:conclusions}. 

\noindent
\textbf{Notation:} We use the notation $[n] := \{1,2,\cdots,n\}$.  For a vector $x \in \real^n$, the $i$-th component is denoted as $x_i$. We use $\mathbb{1}$ to denote the vector of ones, where the dimension is clear from context. Given a $\lambda > 0$, we define $\Delta_\lambda^n := \setdef{x \in \realnonnegative^n}{\sum_{i \in [n]} x_i = \lambda}$. For a matrix $A \in \real^{n \times m}$ the $(i,j)$-th element is denoted as $A_{ij}$. The space of $n \times m$ \textit{column stochastic matrices} is written as ${\colsto{n}{m} = \setdef{A \in \realnonnegative^{n \times m}}{\sum_{i \in [n]}A_{ij} = 1 \text{ for all } j\in [m]}}$.
\section{Model and problem statement} \label{sec:modelandproblem}
Consider a network defined by a directed graph $\GG = (\VV,\EE)$, where $\VV = [N]$, $N \in \mathbb{N}$, is the set of vertices and ${\EE \subseteq \VV \times \VV}$ is the set of edges. Each edge $e_k \in \EE$ consists of an ordered pair of vertices $(\vin_k,\vout_k)$, termed in- and out-vertex respectively, where edge $e_k$ points from $\vin_k$ to $\vout_k$. For $v,w \in \VV$, a path $p$ from $v$ to $w$ is then an ordered set of edges $(e_1,\cdots,e_\ell)$ such that $\vin_{1} = v$, $\vout_{\ell} = w$ and $\vout_{i} = \vin_{i+1}$ for all ${i \in [\ell-1]}$. In addition, paths are defined to be acyclic, meaning that no vertex is visited twice when traveling along a path. To this network we associate an origin $v_o \in \VV$ and a destination ${v_d \in \VV}$. The set of paths in the graph starting at $v_o$ and ending at $v_d$ are collected in the set $\PP$. For notational convenience and without loss of generality, we assume that a unit amount of traffic needs to be routed from the source to the destination. The amount of traffic that uses the path $p \in \PP$ is denoted as $f_p \in \realnonnegative$ and is referred to as the \textit{flow on path} $p$. Taken together, the flows on all different paths give rise to a vector $f \in \realnonnegative^{n}$ which is called a \textit{path-flow}, or simply a \textit{flow}. The set of feasible flows is then given by
\begin{equation*}
	\HH := \setdefB{f \in \realnonnegative^n}{\sum_{p \in \PP} f_p = 1}.
\end{equation*}
Based on this path-flow, the flow over an edge $e_k \in \EE$, denoted $f_{e_k}$, is simply the sum of the flows of all paths containing $e_k$:
\begin{equation} \label{eq:edge-flows-defined}
	f_{e_k} : = \sum_{p \ni e_k} f_p.
\end{equation}
We occasionally denote the vector constituting all edge-flows with $f_{\mathrm{edge}}$. At any instant, the network can be in one of a finite number of states. This can model, for instance, the presence or absence of an accident on a road, or varying weather conditions. The set of states is denoted by $\Theta := \{\theta_1, \cdots, \theta_m\}$. In any state $\theta_s \in \Theta$, each edge $e_k \in \EE$ is associated with a cost function ${\map{C_{e_k}^{\theta_s}}{\realnonnegative}{\realnonnegative}}$, $f_{e_k} \mapsto C_{e_k}^{\theta_s}(f_{e_k})$, which we assume to be known, and continuous and strictly increasing. This function models, for example, the time it takes to traverse edge $e_k$ in state $\theta_{s}$. Given edge-costs, the cost of traversing path $p$ in state $\theta_{s}$ is simply the sum of costs of all edges contained in $p$:
\begin{equation} \label{eq:path-cost-from-edges}
	C^{\theta_{s}}_p(f) = \sum_{e_k \in p} C^{\theta_{s}}_{e_k}(f_{e_k}).
\end{equation}

We consider a Bayesian setting, where the users of the network are assumed to have a \textit{prior} belief $q \in \Delta_1^m$ regarding the probability distribution of the state in which the network operates at any instant. That is, $q_s$ is the probability with which the users believe the network will be in state $\theta_s$, given that they have received no additional information. For $\varphi \in \real^m$, the \textit{weighted-cost} under $\varphi$ of traversing a path $p$ and an edge $e_k$ are respectively given by
\begin{equation} \label{eq:weighted-path-cost}
	\begin{aligned}
		C^\varphi_p(f) \! := \! \! \sum_{s \in [m]} \varphi_s C_p^{\theta_s}(f), \quad
		C^\varphi_{e_k}(f) \! := \! \! \sum_{s \in [m]} \varphi_s C_{e_k}^{\theta_s}(f_{e_k}).
	\end{aligned}
\end{equation}
When $\varphi \in \Delta_1^m$, i.e., when $\varphi$ is a probability distribution, we call these the \textit{expected cost} under $\varphi$. For notational convenience, we define the following:
\begin{align*}
\mathcal{C} &:= \{C_{e_k}^{\theta_s}\}_{e_k \in \EE, s \in [m]},	\\
C_p(f)	&:= \big(C^{\theta_1}_p(f), \cdots, C^{\theta_m}_p(f)\big)^\top,	\\
C^\varphi(f)	&:= \big(C^{\varphi}_1(f), \cdots, C^{\varphi}_n(f)\big)^\top.
\end{align*}
Here, $\CC$ denotes the set of all edge-cost functions, $C_p$ is the vector of the cost functions associated to path $p$ per state, and $C^{\varphi}$ is the vector of weighted-costs under $\varphi$ per path.

For a given a probability distribution $\varphi$ over the states $\Theta$, we assume that the users aim to minimize their own expected cost of traveling, where the expectation is taken with respect to the distribution $\varphi$. To formalize which flows result from such rational decision-making of users, we define the following notion of Wardrop equilibrium: 
\begin{definition}\longthmtitle{$\varphi$-WE} \label{def:qbased-WE}
	Given a set of paths $\PP$, states ${\Theta}$, cost functions $\CC$, and a probability distribution $\varphi \in \Delta^m_1$, a flow $f^\varphi$ is said to be a \emph{$\varphi$-based Wardrop equilibrium} ($\varphi$-WE) if $f^\varphi \in \HH$ and for all $p \in \PP$ such that $f^\varphi_p > 0$ we have
	\begin{equation} \label{eq:def-of-WE-constraint}
	C^\varphi_p(f^\varphi) \leq C_r^\varphi(f^\varphi) \quad \text{for all } r \in \PP.
	\end{equation}
	The set of all $\varphi$-WE is denoted $\WW^\varphi$.
\end{definition}
We will sometimes refer to a $\varphi$-WE $f^\varphi$ as a flow, or a WE, induced by the distribution $\varphi$. The intuition behind this notion of Wardrop equilibrium is that when the flow is in $\varphi$-WE, a single driver cannot decrease her expected cost by changing her routing decision. Note that, under the assumptions on $\CC$ and $\HH$, a flow $f^\varphi$ is a $\varphi$-WE if and only if it is the solution of the \textit{variational inequality} (VI) problem $\VI(\HH,C^\varphi)$. For a given map $F$ and set $\KK$ the associated VI problem $\VI(\KK,F)$ is to find $f^* \in \KK$ satisfying $(f - f^*)^\top F(f^*) \geq 0$ for all $f \in \KK$. A $\varphi$-WE is not necessarily unique. Despite this, the weighted edge-cost under $\varphi$ is the same for any $\varphi$-WE \cite{RC-VD-MS:21}. That is, for any two WE $f^\varphi, \widehat{f}^\varphi \in \WW^{\varphi}$ we have $C^\varphi_{e_k}(f^\varphi_{e_k}) = C^\varphi_{e_k}(\widehat{f}^\varphi_{e_k})$ for all $e_k \in \EE$. Since we assume that functions $C^\varphi_{e_k}$ are strictly increasing, this implies $f^\varphi_{e_k} = \widehat{f}^\varphi_{e_k}$ for all $e_k \in \EE$. In fact, we have that $\widehat{f}^\varphi$ is a $\varphi$-WE if and only if\cite[Chapter 3]{MB-CBM-CBW:56}
\begin{equation}\label{eq:edge-flow-constant}
	\widehat{f}^\varphi_{e_k}= f^\varphi_{e_k} \text{ for all } e_k \in \EE, f^\varphi \in \WW^\varphi.
\end{equation}
Throughout the paper, we denote the unique edge-flow on edge $e_k$ under all $\varphi$-WE with $f^\varphi_{e_k}$.

The last part of the model is a traffic information system (TIS), that observes the state $\theta_s$ of the network at any instant, and subsequently supplies information about this state to the drivers. The TIS has a set of signals $\mathcal{Z} := \{\zeta^1,\cdots,\zeta^z\}$ from which it chooses one to send to the users at any instant of the game. Before the traffic is routed, the TIS commits to a signalling scheme $\Phi: \Theta \mapsto \Delta_1^z$. Each state $\theta_s$ is mapped by $\Phi$ to a probability vector $\Phi(\theta_s) := \phi^{\theta_s} \in \Delta_1^z$.
After observing state $\theta_s$, the TIS randomly draws a signal from $\mathcal{Z}$ to send to the participants, where the probability of sending signal $\zeta^u$ is given by the $u$-th element of $\phi^{\theta_s}$. In our setting all participants receive the same signal, which is known as \textit{public signalling}. Note that the signalling scheme $\Phi$ can be represented as a ${z \times m}$ column stochastic matrix; that is, $\Phi \in \colsto{z}{m}$, with the $(u,s)$-th entry, denoted $\phi^u_s$, 
giving the probability of sending signal $\zeta^u$ after observing state $\theta_s$. We will adhere to this matrix representation of $\Phi$ throughout the paper.

After receiving a signal $\zeta^u$, the users update their belief about the state of the network by forming a \textit{posterior} $\widetilde{q}$ using Bayes' rule:
\begin{equation} \label{eq:Bayes-rule}
	\begin{split}
		\widetilde{q}_s = \mathbb{P}[\theta_s|\zeta^u] = \frac{\mathbb{P}[\zeta^u|\theta_{s}]q_s}{\sum_{\ell \in [m]} \mathbb{P}[\zeta^u|\theta_{\ell}]q_\ell}	= \frac{\phi^{u}_s q_s}{\sum_{\ell \in [m]} \phi^{u}_{\ell} q_{\ell}},	\\	
	\end{split}
\end{equation}
for all $s \in [m]$, where $\mathbb{P}[\theta_s|\zeta^u]$ is the probability of the network being in state $\theta_s$ having received the signal $\zeta^u$ and $\mathbb{P}[\zeta^u|\theta_s$] is the probability of sending signal $\zeta^u$ after observing state $\theta_{s}$.
The resulting flow is then assumed to be a $\widetilde{q}$-based Wardrop equilibrium. When no additional information regarding the state of the network is available to the users, the flow is assumed to depend on the prior $q$, and is given by a $q$-WE denoted as $f^q$. Throughout this paper we will use $q$ to denote the prior, $\widetilde{q}^{\zeta^u}$ to denote the posterior with respect to the signal $\zeta^u$, and use $\widetilde{q}$ when the signal is clear from the context. Associated sets of WE will be denoted as $\WW^q$, $\WW^{\zeta^u}$, and $\WW^{\widetilde{q}}$, respectively. Similarly, given a distribution $\varphi \in \Delta_1^m$ we will use the notation $\WW^\varphi$ for the set of $\varphi$-based WE, and $\WW^{\theta_s}$ for a $\varphi^{\theta_s}$-based WE, where the distribution $\varphi^{\theta_s}$ is defined by $\varphi^{\theta_s}_s = 1$. Note that when $q_s = 0$ for some $s \in [m]$, despite the TIS observing state $\theta_s$, it is possible that $\widetilde{q}_{\ell}$ is ill-defined for some $\ell \in [m]$ as it may involve division by zero. To avoid this issue, we assume that $q_s > 0$ for all $s \in [m]$.
\subsection{Main idea and motivating example} 
The setup we have introduced here is an adaptation of the Bayesian persuasion framework, as introduced in \cite{EK-MG:11}, to routing games. This model has received growing attention in recent years, often focused on the analysis and derivation of signalling schemes that maximize social welfare \cite{SD-EK-RM:17,YZ-KS:18,OM-CL:19}. A common assumption in this context is that the prior belief of the population is known to the TIS, see e.g., \cite{YZ-KS:18,MW-SA-AEO:21}. However, due to lack of information accurate estimation of the true distribution may not be possible for the participants. In addition a population as a whole may suffer from biases, such as a tendency to favor highways over local roads. Thus the prior of a population may differ from the true distribution and remain unknown to the TIS. Note that the TIS can observe the state directly in each instance, and can therefore form a relatively accurate estimate of the true distribution as a result. In this paper we aim to show how a TIS can gain information about the prior by observing equilibrium flows. We now briefly discuss a motivating example showing how for a TIS that aims to design a signalling scheme to minimize social cost, a mismatch between the prior and the estimate of that prior made by that TIS can lead to an increase in social cost. 
\begin{example}\longthmtitle{Motivating example}\label{ex:prior}
	{\rm 
		Consider a network with two nodes, the origin $v_o$ and destination $v_d$, and two parallel paths going from $v_o$ to $v_d$ as depicted in Figure \ref{fig:network_unknown_prior}.
		\begin{figure}[t]
			\centering
			\subfloat[The example network.]{\includegraphics[width = 0.1\textwidth]{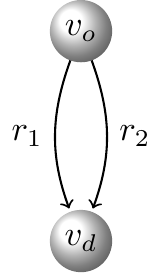}\label{fig:network_unknown_prior}} \hspace{3 pt}
			\subfloat[Cost for TIS with and without exact knowledge of~$q$.]{\includegraphics[width = 0.37\textwidth]{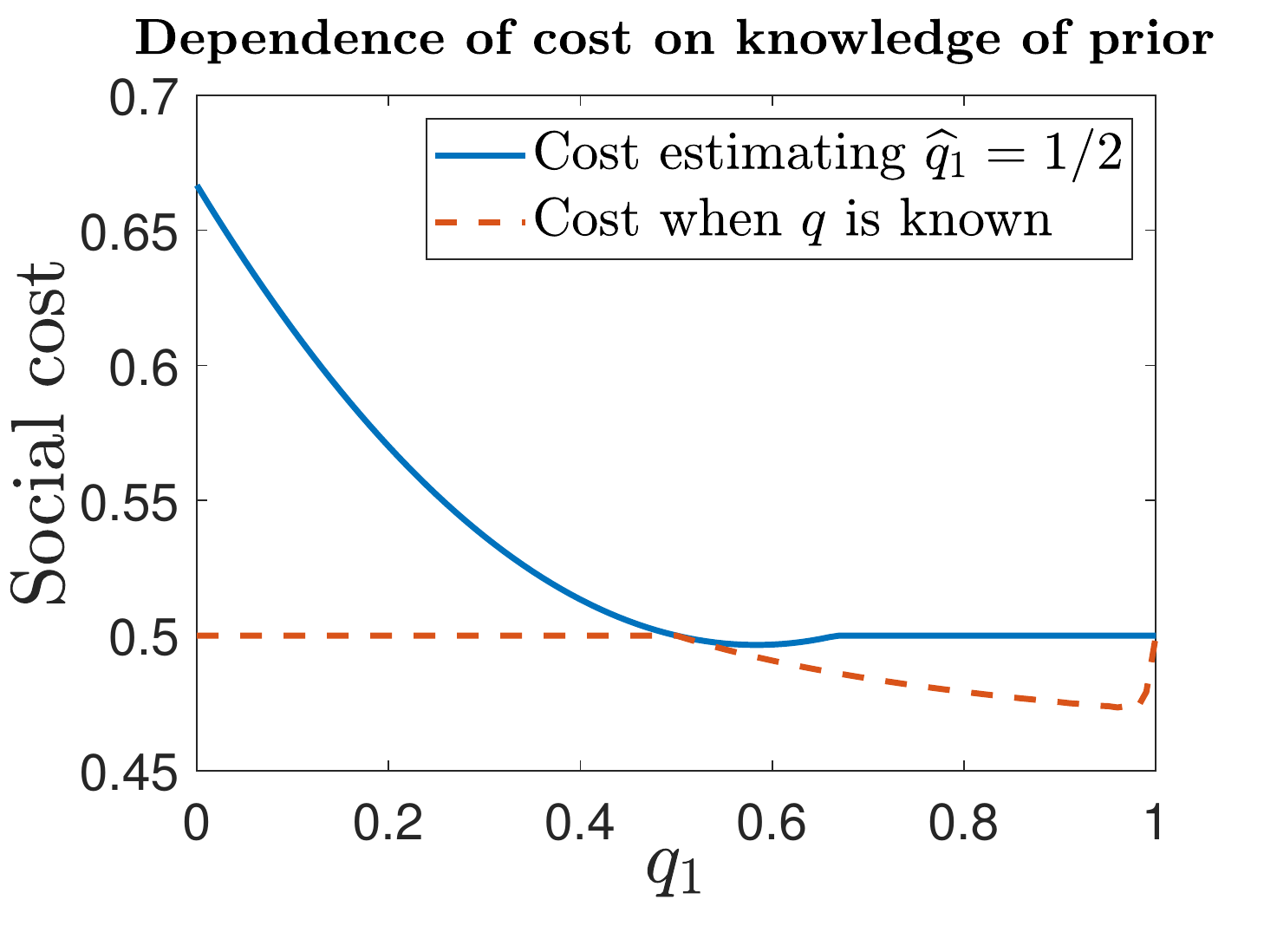}\label{fig:cost_unknown_prior}}
			\caption{An illustration (see Example~\ref{ex:prior}) of how wrongly estimating the prior $q$ can affect the ability of the TIS to minimize social cost.}
		\end{figure}
		The network can be in two states, and the cost functions of the paths in these states are
		\begin{align*}
			& C_1^{\theta_1}(f) = C_1^{\theta_2}(f) = 2f_1 + \frac{1}{2}, 
			\\
			& C_2^{\theta_1}(f) = 0, \text{ and } C_2^{\theta_2}(f) = 1.
		\end{align*}
		The probability distribution of states $\theta_1$ and $\theta_2$ is given by ${\varphi^{\tr} = (\varphi^\tr_1, \varphi^\tr_2)}$, where $\theta_1$ occurs with probability $\varphi^\tr_1 = 0.5$ and $\theta_2$ occurs with probability $\varphi^\tr_2 = 1 - \varphi^\tr_1$. The distribution $\varphi^\tr$ is assumed to be known to the TIS. The goal of the TIS is to minimize the long-term average social cost, which is a function of the signalling scheme. For a general network, given the prior belief $q$, the state $\theta_{s}$, and a message $\zeta^u$, the incurred social cost is given by
		\begin{align} \label{eq:cost-per-stage}
			J_q^{\stage}(\zeta^u,\theta_{s}) := \sum_{p \in \PP}\widetilde{f}_p^{\zeta^u}C_p^{\theta_{s}}(\widetilde{f}^{\zeta^u}),
		\end{align}
		where $\widetilde{f}^{\zeta^u}$ is a $\widetilde{q}^{\zeta^u}$-WE. Note that for any two $\widetilde{q}^{\zeta^u}$-WE, say $\widetilde{f}^{\zeta^u}$ and $\widehat{f}^{\zeta^u}$, we have $\widetilde{f}^{\zeta^u}_{e_k} = \widehat{f}^{\zeta^u}_{e_k}$ for all $e_k \in \EE$. From~\eqref{eq:path-cost-from-edges} and~\eqref{eq:edge-flows-defined}, we then conclude that~\eqref{eq:cost-per-stage} is independent of the choice of $\widetilde{q}^{\zeta^u}$-WE.
		The long-term average cost will be the sum of $J_q^{\stage}(\zeta^u,\theta_{s})$ over all possible combinations of signals $\zeta^u$ and states $\theta_s$, weighted by the probability $\varphi^\tr_s$ that $\theta_{s}$ occurs, and the probability $\phi_s^u$ of signal $\zeta^u$ being send when $\theta_{s}$ occurs. This will therefore depend on the number of signals that the scheme employs. However, in \cite[Proposition 3]{OM-CL:19} it is shown that a public signalling scheme $\Phi \in \colsto{z}{m}$ needs no more than $m$ signals to achieve the optimum and therefore we set $z = m$. Summarizing this, the long-term average cost that the TIS aims to minimize is given by
		\begin{align*}
			J_q(\Phi): = \sum_{s \in [m]} \sum_{u \in [m]} \varphi^\tr_s \phi^u_s \sum_{p \in \PP}\widetilde{f}_p^{\zeta^u}C_p^{\theta_{s}}(\widetilde{f}^{\zeta^u}).
		\end{align*}
		When the TIS knows the prior belief $q$, it aims to find a scheme $\Phi$ that minimizes $J_q(\Phi)$. When the TIS does not know this prior belief, it assumes it to be same as the the probability distribution of states $\varphi^\tr$, and therefore employs a signalling scheme $\Phi$ that minimizes $J_{\varphi^\tr}(\Phi)$. Whenever $q \neq \varphi^\tr$, designing a signalling scheme using $\varphi^\tr$ as an estimate of $q$ can increase the social cost. This we show in Figure~\ref{fig:cost_unknown_prior}. The horizontal axis in the plot depicts the prior held by the users and since we only consider two states, it is completely specified by the first component $q_1$ of the two-dimensional vector $q$. The blue line shows the long-term average cost of the game when the TIS uses $\varphi^\tr$ as an estimate of the prior $q$ and employs a signalling scheme that minimizes $J_{\varphi^\tr}$. The dashed orange line shows the cost achieved when the TIS uses the exact knowledge of $q$ and employs an optimal signalling scheme minimizing $J_q$. We see that the TIS with full knowledge performs better, with the difference becoming more pronounced as $q$ moves further away from $\varphi^\tr$.
		\oprocend
	}
\end{example}
The above example highlights the need for the TIS to accurately know the prior in order to find the optimal scheme minimizing social cost. Motivated by this fact, the following section discusses how observations of Wardrop equilibria, and knowledge of the signalling scheme $\Phi$ can be used by the TIS to infer the prior $q$.
\section{Inferring the prior: General case}\label{sec:general}
In this section we study how observations of Wardrop equilibria can be used to obtain information about the prior. Section~\ref{ssec:dist-eq} investigates the relationship between path-flows under $\varphi$-WE, edge-flows under $\varphi$-WE, and the distribution $\varphi$. In Section~\ref{ssec:existence} we use the gained insights to show that under very mild conditions, there always exists a signalling scheme allowing for the identification of the prior. After this, in Section \ref{ssec:design} we give a procedure for sequentially updating a given signalling scheme in order to find a scheme allowing for the identification of the prior. At the end of Section \ref{ssec:design} we provide an illustrative example.

We will start by briefly introducing the intuition behind the ideas presented in this section. Consider the case where the TIS provides no information to the users and the resulting flow $f^q$ is 
a $q$-WE as players base their routing choices on the prior.\footnote{The same situation can be achieved by using a signalling scheme which supplies no information, for instance by setting $\phi^u_s = \frac{1}{z}$ for all $u,s$.} From Definition \ref{def:qbased-WE} we know that $f^q$ satisfies \eqref{eq:def-of-WE-constraint} where $\varphi$ is replaced by $q$. That is,
\begin{equation} \label{eq:VI-cond-specific}
	\begin{split}
		C^q_p(f^q) & \! = \! C_r^q(f^q), \quad \forall p,r \! \in \! \PP \text{ such that } f^q_p,f^q_r \! > \! 0,
		\\
		C^q_p(f^q) &\! \leq \! C_r^q(f^q), \quad \forall p,r \! \in \! \PP \text{ such that } f^q_p \! > \! 0, f^q_r \! = \! 0.
	\end{split}
\end{equation}
Defining the matrix-valued map $C^{\texttt{mat}}:\HH \to \realnonnegative^{n \times m}$ as
\begin{align*}
	C^{\texttt{mat}}(f) := \left(\begin{array}{cccc} C_1^{\theta_1}(f) & C_1^{\theta_2}(f) & \dots & C_1^{\theta_m}(f) 
		\\
		C_2^{\theta_1}(f) & C_2^{\theta_2}(f) & \dots & C_2^{\theta_m}(f) \\
		\vdots & \vdots & & \vdots  \\
		C_n^{\theta_1}(f) & C_n^{\theta_2}(f) & \dots & C_n^{\theta_m}(f)
	\end{array}\right),
\end{align*}
we have $C^{q}(f^q) = C^{\texttt{mat}}(f^q) \! \cdot \! q$ and so \eqref{eq:VI-cond-specific} can be rewritten as
\begin{equation} \label{eq:VI-cond-specific2}
	\begin{split}
		\big(C_p^{\texttt{mat}}(f^q) \! - \! C_r^{\texttt{mat}}(f^q)\big) q \! &= \! 0, \,  \forall p,r \text{ with } f^q_p,f^q_r \! > \! 0,
		\\
		\big(C_p^{\texttt{mat}}(f^q) \! - \! C_r^{\texttt{mat}}(f^q)\big) q \! &\leq \! 0, \,  \forall p,r \text{ with } f^q_p \! > \! 0, f^q_r \! = \! 0,
	\end{split}
\end{equation}
where $C_p^{\texttt{mat}}(f)$ denotes the $p$-th row of $C^{\texttt{mat}}(f)$. Given a $q$-WE $f^q$, the above gives constraints on the possible values that the prior can take. In this way the equilibrium flow $f^q$ can help us in identifying the prior.
Most information can be obtained from the equality constraints, though it is also possible that a combination of equality and inequality constraints together result in additional equality constraints. In addition to the above, we also have the constraint ${\sum_{s \in [m]}q_s = 1}$, which is linearly independent from all equality constraints obtained from \eqref{eq:VI-cond-specific2}\footnote{An intuitive way to see this is as follows. When $f^q$ is fixed, for any $q$ that satisfies the constraints in \eqref{eq:VI-cond-specific}, $cq$ will also satisfy these constraints for any $c \in \realnonnegative$. This is clearly not the case for the constraint $\sum_{s \in [m]}q_s = 1$}. In this way we find a number of linearly independent equality constraints on $q$. Since $q \in \real^m$ we need $m$ such constraints in order to uniquely determine $q$. If the flow $f^q$ does not allow us to determine $q$ uniquely, we can use a public signalling scheme $\Phi$ to induce different posteriors. These posteriors will lead to different equilibrium flows resulting in equality constraints of the form \eqref{eq:VI-cond-specific2}, where $q$ and $f^q$ are replaced with $\widetilde{q}$ and $f^{\widetilde{q}}$, respectively. Using \eqref{eq:Bayes-rule}, these constraints on the posterior $\widetilde{q} = \widetilde{q}^{\zeta^u}$ can be rewritten into constraints on the prior $q$, by noting that
\begin{align*}
	\big(C_p^{\texttt{mat}}(f) -  C_r^{\texttt{mat}}(f)\big) \widetilde{q}^{\zeta^u}	= \sum_{s \in [m]} \frac{\phi^u_s \big(C_p^{\theta_{s}}(f) - C_r^{\theta_{s}}(f)\big) q_s}{\sum_{\ell \in [m]} \phi^{u}_{\ell} q_{\ell}} .
\end{align*}
Thus constraints on the prior $q$ imposed by observing the equilibrium flow $\widetilde{f}^{\zeta^u}$ are of the form
\begin{subequations}\label{eq:posterior-WE-constraint}
	\begin{align}
		&\sum_{s \in [m]} \! \phi^u_s \big(C_p^{\theta_{s}}(\widetilde{f}^{\zeta^u}) \! - \! C_r^{\theta_{s}}(\widetilde{f}^{\zeta^u})\big)q_s \! = \! 0, \label{eq:posterior-WE-constraint-1}
		\\
		&\sum_{s \in [m]} \! \phi^u_s \big(C_p^{\theta_{k}}(\widetilde{f}^{\zeta^u}) \! - \! C_r^{\theta_{s}}(\widetilde{f}^{\zeta^u})\big)q_s \! \leq \! 0,  \label{eq:posterior-WE-constraint-2}
	\end{align}
\end{subequations}
where~\eqref{eq:posterior-WE-constraint-1} holds for all $p,r$ with $\widetilde{f}^{\zeta^u}_p,\widetilde{f}^{\zeta^u}_r  >  0$ and~\eqref{eq:posterior-WE-constraint-2} holds for all $p,r$ such that $\widetilde{f}^{\zeta^u}_p  >  0$ and $\widetilde{f}^{\zeta^u}_r  =  0$.
In the above conditions, the denominator has been dropped, since it is the same for each term in the summation, and assumed to be positive. For a signalling scheme $\Phi$ we denote the set of all priors satisfying all obtained constraints from all signals as 
\begin{align*}
	\QQ_{\Phi} = 
	\setdef{q \in \Delta_1^m}{ q \text{ satisfies \eqref{eq:posterior-WE-constraint}} \text{ for all } u \in [z]}.
\end{align*}
The above seems to depend on which specific $\widetilde{q}^{\zeta^u}$-WE $\widetilde{f}^{\zeta^u}$ are observed. However it follows from upcoming results, specifically Corollary \ref{cor:Q-independent-of-specific-flow}, that this is not the case. We give the following definition:
\begin{definition}\longthmtitle{$q$-identifying signalling scheme}
	Given a set of paths $\PP$, states $\Theta$, cost functions $\CC$, and a prior $q \in \Delta^m_1$, a signalling scheme $\Phi \in \colsto{s}{m}$ is called $q$-identifying if $\QQ_{\Phi} = \{q\}$.
\end{definition}
The main focus of this paper is addressing the question ``How can we design $\Phi$ so as to ensure that it is $q$-identifying?''
Before we can discuss this however, we will first investigate the relations between the distribution $\varphi$, the associated $\varphi$-WE $f^{\varphi}$, and the related edge-flows $f^\varphi_{e_k}$.
\subsection{Probability distribution and equilibrium} \label{ssec:dist-eq}
The results in upcoming sections build upon three lemma's presented here, which give insight in how the edge-flows under $\varphi$-WE, path-flows under $\varphi$-WE, and the distribution $\varphi$ relate to each other. To ease the exposition of the first lemma, we introduce the following notation:
\begin{equation*}
	\HH_e \! := \! \setdef{v \in \realnonnegative^{\abs{\EE}}}{\exists f \in \HH \text{ such that } v_k = f_{e_k} \enskip \forall k \! \in \! [\abs{\EE}]}.
\end{equation*}
Here, $f_{e_k}$ is defined by \eqref{eq:edge-flows-defined}. Note that since $\HH$ is compact, so is $\HH_e$. We use this set in the proof of the following result, which shows that the edge-flows under $\varphi$-WE change continuously with respect to $\varphi$.
The first of these results shows that the edge-flows under $\varphi$-WE change continuously with respect to $\varphi$.
\begin{lemma} \longthmtitle{Continuity of $\varphi$-WE edge-flows} \label{lem:continuity-edge-flow}
	Let $\PP$, $\Theta$, and $\CC$ be given. For every $\epsilon > 0$, there exists a $\delta > 0$ such that for any two distributions $\varphi, \xi \in \Delta_1^m$, we have
	\begin{equation*}
		\norm{\varphi - \xi} < \delta \Rightarrow \abs{f^\varphi_{e_k} - f^\xi_{e_k}} < \epsilon \quad \forall e_k \in \EE.
	\end{equation*}
	In other words, the edge-flows under $\varphi$-WE depend continuously on the distribution $\varphi$.
\end{lemma}
\begin{proof}
	For $\varphi \in \Delta_1^m$, recall the notation of $C^{\varphi}_{e_k}$ from~\eqref{eq:weighted-path-cost}. Following~\cite{MB-CBM-CBW:56}, a flow vector $f^\varphi \in \HH$ is a $\varphi$-WE if and only if it is a solution of the following optimization problem:
	\begin{equation}\label{eq:opt-path}
		\min_{f \in \HH} \, \,  \sum_{e_k \in \EE} \int_0^{f_{e_k}} C_{e_k}^\varphi(t)dt,
	\end{equation}
	where for a path-flow $f$, the quantity $f_{e_k}$ is the corresponding flow on edge $e_k$ given by~\eqref{eq:edge-flows-defined}. Recall from~\cite{RC-VD-MS:21} that while the $\varphi$-WE need not be unique, the edge-flows induced by them are. Thus, following~\eqref{eq:opt-path}, the edge-flows associated to $\varphi$-WE are given by the unique solution of the following problem:
	\begin{equation} \label{eq:Beckmann}
		\min_{v \in \HH_e} \sum_{e_k \in \EE} \int_0^{v_k} \sum_{s \in [m]} \varphi_s C^{\theta_{s}}_{e_k}(t)dt.
	\end{equation}
	Consider the above optimization problem with $\varphi$ as a parameter. Given $\varphi$, denote the optimal solution as $f^\varphi_{\mathrm{edge}}$. Since the objective function of the above problem depends linearly on $\varphi$ and the domain is compact and independent of $\varphi$, we deduce from~\cite[Proposition 4.4]{JFB-ASH:2013} that the map $\varphi \mapsto f^\varphi_{\mathrm{edge}}$ is continuous. This concludes the proof.
\end{proof}
To ease the exposition of the next result, we define
\begin{equation*}
	\RR^{\mathrm{use}}_\varphi := \setdef{p \in \PP}{\exists f^\varphi \in \WW^\varphi \text{ such that } f^\varphi_p > 0}.
\end{equation*}
That is, $\RR^{\mathrm{use}}_\varphi$ denotes the set of all paths $p$ for which there exists a $\varphi$-WE such that a positive amount of flow is routed onto path $p$. We call these paths the \emph{used} paths. The set of $\varphi$-WE then has the following useful properties:
\begin{lemma}\longthmtitle{Characterizing used paths of $\varphi$-WE}\label{lem:used-set-properties}
	Let $\PP$, $\Theta$, $\CC$, and $\varphi \in \Delta_1^m$ be given. We have the following:
	\begin{enumerate}
		\item There exists an $f^\varphi \! \in \! \WW^\varphi$ satisfying $f_p^\varphi \! > \! 0$ for all ${p \! \in \! \RR^{\mathrm{use}}_\varphi}$.
		\item We have ${p \in \RR^{\mathrm{use}}_\varphi}$ if and only if $f^\varphi_{e_k} > 0$ for all $e_k \in p$.
	\end{enumerate}
\end{lemma}
\begin{proof}
	The set $\WW^\varphi$ is convex. This can be deduced from~\eqref{eq:edge-flow-constant} and noting that if two path flows induce the same edge flow, then any convex combination of these flows will still induce that edge flow. The first claim follows from convexity of $\WW^\varphi$. To see the complete reasoning, denote first for any $r \in \RR^{\mathrm{use}}_\varphi$, a flow $f^{\varphi,r} \in \WW^\varphi$ as a WE flow where $f^{\varphi,r}_r > 0$. Such a flow exists by the definition of $\RR^{\mathrm{use}}_\varphi$. Next select scalars $c_r > 0$ for all $r \in \RR^{\mathrm{use}}_\varphi$ such that $\sum_{r \in \RR^{\mathrm{use}}_\varphi} c_r = 1$. Using the selected WE flows and scalars, define $f^{\mathrm{use}} := \sum_{r \in \RR^{\mathrm{use}}_\varphi} c_r f^{\varphi,r}$. Note that $f^{\mathrm{use}} \in \WW^\varphi$ as this set is convex. Finally, by definition of $\{f^{\varphi,r},c_r\}$ and the fact that all WE flows are nonnegative, we deduce that $f^{\mathrm{use}}_r > 0$ for all $r \in \RR^{\mathrm{use}}_\varphi$. This establishes the first claim.
	
	For the second claim, the ``only if'' part is easier to deduce. Let $p \in \RR^{\mathrm{use}}_\varphi$ and let $f^{\varphi,p} \in \WW^\varphi$ satisfy $f^{\varphi,p}_p > 0$. Since $f^{\varphi,p}_r \geq 0$ for all $r \in \PP$, it follows from~\eqref{eq:edge-flows-defined} that $f^{\varphi,p}_{e_k} > 0$ for all $e_k \in p$. For the other direction, we provide a sketch of arguments here in the interest of space. First, we note that for a $\varphi$-WE, a total flow of unity enters and leaves the network at the origin and destination, respectively, while for all other vertices the flow satisfies mass-conservation constraints. That is, the total flow entering and leaving a vertex are equal. Second, it can also be shown that $\varphi$-WE does not contain any cycle with a positive amount of flow on all its edges. To see this, note that reducing the flow equally from all edges in such a cycle will preserve mass conservation and inflow and outflow constraints, while the value of \eqref{eq:Beckmann} decreases. Thus, with the presence of a positive-flow cycle, the path-flow can not be a $\varphi$-WE. Lastly, consider any path $p$ such that $f^\varphi_{e_k} > 0$ for all $e_k \in p$. Set $f_p^{\varphi} := \min_{e_k \in p} f^\varphi_{e_k}$ and then subtract $f_p^{\varphi}$ of flow from all edges in $p$. The new flow will then still satisfy mass-conservation constraints, but the inflow and outflow at the origin and destination have both decreased by $f_p^{\varphi}$. Continue this procedure until all flow has been assigned
	and the result is a feasible flow $f^{\varphi}$ which induces the same edge-flow as any $\varphi$-WE. Therefore, $f^{\varphi}$ is a WE, and it satisfies $f^{\varphi}_p > 0$ for any desired $p \in \RR^{\mathrm{use}}_{\varphi}$ by construction, which concludes the proof. The procedure of assigning flow is treated in more detail in \cite[Theorem 2.1]{RKA-TLM-JBO:88}.
\end{proof}
In the next result we show that for a given $f \in \HH$, the set of all distributions $\varphi$ such that $f \in \WW^\varphi$ is compact and convex. 
\begin{lemma}\longthmtitle{Convexity of set of distributions inducing the same $\varphi$-WE} \label{lem:interval-of-constant-WE}
	Let $\PP$, $\Theta$, $\CC$, and $f \in \HH$ be given. The set of distributions $\varphi$ with $f \in \WW^\varphi$ is compact and convex.   
\end{lemma}
\begin{proof}
	For any distribution $\varphi \in \Delta_1^m$, we have $f \in  \WW^\varphi$ if and only if the constraints in \eqref{eq:VI-cond-specific} hold, where $q$ and $f^q$ are replaced with $\varphi$ and $f$, respectively. Since $f$ is fixed, the map $C^{\texttt{mat}}(f)$ is also fixed and we see that \eqref{eq:VI-cond-specific} imposes a number of equality and non-strict inequality constraints on $\varphi$, all of which are affine. Therefore, the set of $\varphi$ satisfying these constraints is convex and closed. Since distributions belong to a compact set $\Delta_1^m$, the claim follows.
\end{proof} 
We illustrate the implications of Lemma~\ref{lem:interval-of-constant-WE} using the following examples. For simplicity, we have chosen examples such that the $\varphi$-WE are unique.
\begin{example} \longthmtitle{Demonstration of Lemma~\ref{lem:interval-of-constant-WE}}\label{ex:constant-WE}
	{\rm
		Consider a 2-path, 2-state network, with cost functions given by
		\begin{equation*}
			\begin{aligned}
				&C_1^{\theta_1}(f) = 0.8 f_2 + 0.7,	&C_1^{\theta_2}(f) = 0.1 f_2 + 0.2,	\\
				&C_2^{\theta_1}(f) = 0.3 f_1 + 0.2,	&C_2^{\theta_2}(f) = 0.5 f_1 + 0.5.
			\end{aligned}
		\end{equation*}
		Figure \ref{fig:2-road-2-state} shows the relationship between the $\varphi$-WE and the distribution $\varphi$.
		\begin{figure}[t]
			\centering
			\subfloat[A graph showing the relationship between the $\varphi$-WE and the distribution $\varphi$ for a 2-path, 2-state scenario.]{\includegraphics[width = 0.35\textwidth]{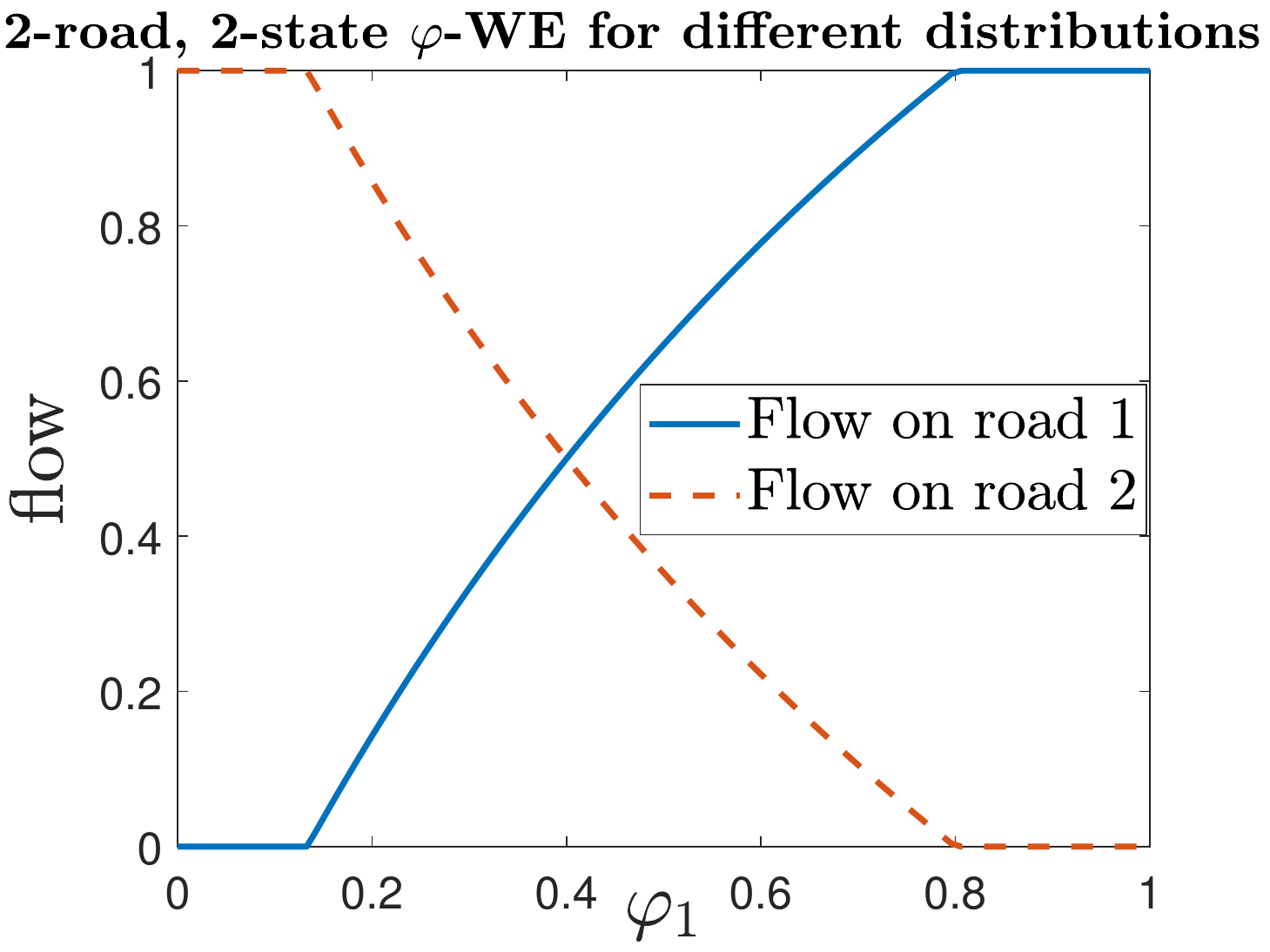}\label{fig:2-road-2-state}}
			\hfill
			\subfloat[A graph showing the relationship between the $\varphi$-WE and the distribution $\varphi$ for a 4-path, 2-state scenario.]{\includegraphics[width = 0.35\textwidth]{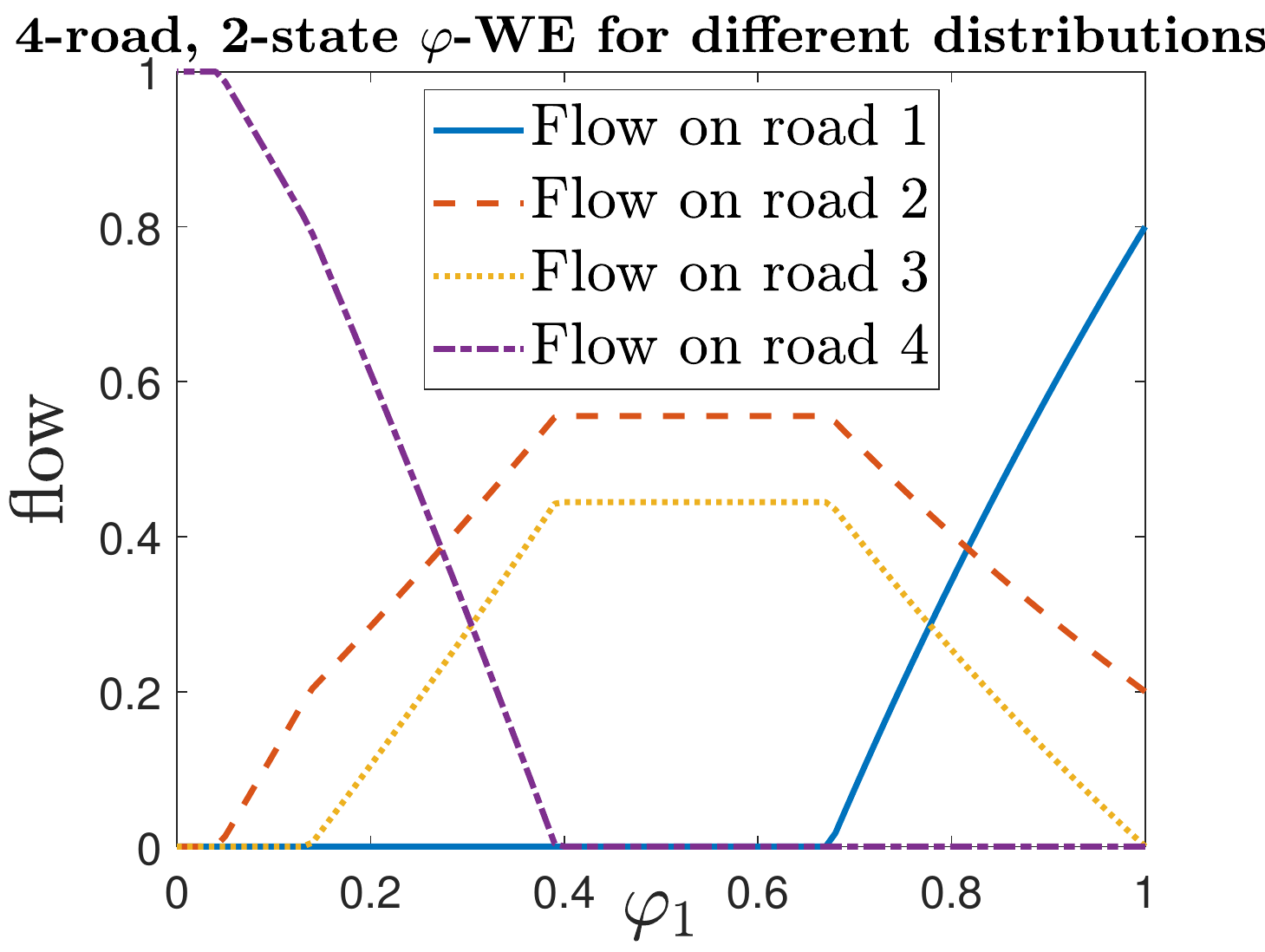}\label{fig:4-road-2-state}}
			\caption{Plots illustrating the regions of constant $\varphi$-WE for cases discussed in Example~\ref{ex:constant-WE}.}
		\end{figure}
		Note that ${\varphi_2 = 1 - \varphi_1}$. Therefore a distribution $\varphi$ is defined completely by $\varphi_1$. The figure shows that the $\varphi$-WE remains constant in two convex regions, namely when $\varphi_1 \leq 0.133$ and when ${\varphi_1 \geq 0.8}$. In one of these cases we have ${f^\varphi = f^{\theta_1} = (1,0)^\top}$ and in the other ${f^\varphi = f^{\theta_2} = (0,1)^\top}$. 
		
		Next we consider a 4-path, 2-state network, with the following cost functions:
		\begin{equation} \label{eq:4,2-example-cost}
			\begin{aligned}
				&C_1^{\theta_1}(f) =  f_1 + 1,	&C_1^{\theta_2}(f) = 0.4 f_1 + 4, \hspace{8 pt}	\\
				&C_2^{\theta_1}(f) = 0.5 f_2 + 1.7, &C_2^{\theta_2}(f) = 0.5 f_2 + 1.7,	\\
				&C_3^{\theta_1}(f) = 0.4 f_3 + 1.8,	&C_3^{\theta_2}(f) = 0.4 f_3 + 1.8,	\\
				&C_4^{\theta_1}(f) = 0.4 f_4 + 3.5,	&C_4^{\theta_2}(f) = 0.6 f_4 + 1. \hspace{8 pt}
			\end{aligned}
		\end{equation}
		Figure \ref{fig:4-road-2-state} shows the dependency between the distribution and the WE. We see that the situation has changed compared to the 2-path, 2-state case. Here we find a region in which the $\varphi$-WE remains constant while not being equal to $f^{\theta_1}$ or $f^{\theta_2}$ or having all flows on one path. We do see that the sets of distributions in which the $\varphi$-WE remains constant are convex, which is in line with Lemma \ref{lem:interval-of-constant-WE}
		
		Although it is perhaps not directly apparent from Lemma~\ref{lem:interval-of-constant-WE}, a consequence of that result is that for any distribution which is not in a convex set where the $\varphi$-WE remains constant, the associated $\varphi$-WE is unique to that distribution. When such a flow is observed, we can derive the unique distribution which induced it. If a $\varphi$-WE is observed that can be induced by multiple distributions, we can at best limit the distribution that induced the flow to a set. Thus, regions of $\Delta_1^m$ where the $\varphi$-WE remains constant are less helpful in identifying $q$, and should be avoided when attempting to design a $q$-identifying signalling scheme.
	} \oprocend
\end{example}
\subsection{Existence of a sufficient signalling scheme} \label{ssec:existence}
Here we discuss the existence of signalling schemes that allow the TIS to identify $q$. Our strategy involves first showing existence for the simplified case where $\Theta = \{\theta_1,\theta_2\}$. We later use this result for the more general case $\Theta = \{\theta_1,\cdots, \theta_m\}$, $m \in \mathbb{N}$ by designing our signalling scheme $\Phi$ in such a way that the resulting posteriors only assign positive probability to exactly two states. A key element of designing such a scheme is the set of flows that provide information regarding the distribution that induced it. In particular,  for the case $\Theta = \{\theta_1,\theta_2\}$, we define the set of \emph{informative} flows $\FF_{\info}$ as follows: 
\begin{equation} \label{eq:F-inf}
	\begin{aligned}
		\FF_{\info}   :=   \setdef{  f \!  \in \!  \HH  }{ f \! \notin \! \WW^{\theta_{1}} \! \cup \! \WW^{\theta_{2}}, \enskip f_p \! > \! 0 \enskip \forall p \! \in \! \RR^{\mathrm{use}}_{\varphi^{\theta_{1}}}}.
	\end{aligned}
\end{equation}
That is, $\FF_{\info}$ is the set of all flows that are not in the set of $\varphi^{\theta_{1}}$- or $\varphi^{\theta_{2}}$-WE, but which do contain a positive amount of flow on all paths that have a positive amount of flow for some $\varphi^{\theta_{1}}$-WE. The importance of this set lies in the fact that for the two-state case, observing a flow from this set allows us to uniquely identify which distribution induced that flow.
\begin{remark}\longthmtitle{Sufficiency of $\FF_{\info}$}
	\rm{We note that it is not necessary for a flow $f^{\varphi}$
		to lie in $\FF_{\info}$ in order to allow $\varphi$ to be identified. Any flow $f^\varphi$ that can only be induced by a unique distribution $\varphi$ will, when observed, necessarily allows us to identify the distribution $\varphi$ that induced it, while the set $\FF_{\info}$ limits the attention to flows with a special relation to the flows in $\WW^{\theta_{1}}$. The set $\FF_{\info}$ is however of special import in the coming results because under mild assumptions, we can identify conditions that allow the flow induced by a signal to be contained in $\FF_{\info}$. } \oprocend
\end{remark}
Before we move on to the results, we collect two useful properties of $\varphi$-WE here, both of which follow from that fact that a flow is a $\varphi$-WE if and only if it induces the same unique edge flow as all other $\varphi$-WE, as mentioned before in~\eqref{eq:edge-flow-constant}. The first property implies that for two distributions, the induced sets of WE overlap if and only if they are equal.
\begin{lemma}\longthmtitle{Intersection of sets of WE induced by two distributions}\label{lem:equality-and-intersection}
	Let $\PP$, $\Theta$, $\CC$, and two distributions $\varphi, \xi \in \Delta_1^m$ be given. Then, $\WW^\varphi \neq \WW^\xi$ if and only if $\WW^\varphi \cap \WW^\xi = \emptyset$.
\end{lemma}
The second property is that for two flows which are both WE induced by the same distribution, the sets of distributions for which these flow are a WE, respectively, are equal.
\begin{lemma}\longthmtitle{Equality of sets of distributions inducing two $\varphi$-WE} \label{lem:equality-of-inducing-sets}
	Let $\PP$, $\Theta$, and $\CC$ be given. For $\xi \in \Delta_1^m$, if we have $f^\xi \in \WW^\xi$ and $\widetilde{f}^\xi \in \WW^\xi$, then
	\begin{equation*}
		\setdef{\varphi \in \Delta_1^m}{f^\xi \in \WW^\varphi} = \setdef{\varphi \in \Delta_1^m}{\widetilde{f}^\xi \in \WW^\varphi}.
	\end{equation*}
\end{lemma}
A useful consequence of the above is that $\QQ_{\Phi}$ is independent of which $\widetilde{q}^{\zeta^u}$-WE flow is observed for each signal $\zeta^u$. 
\begin{corollary}\longthmtitle{Equal informativity of all $q^{\zeta^u}$-WE.} \label{cor:Q-independent-of-specific-flow}
	Let $\PP$, $\Theta$, $\CC$, a signalling scheme ${\Phi \in \colsto{z}{m}}$ and a signal $\zeta^u$ be given. For any ${\widetilde{f}^{\zeta^u}, \widehat{f}^{\zeta^u} \in \WW^{\zeta^u}}$ the set of all priors $\varphi \in \Delta_1^m$ satisfying \eqref{eq:posterior-WE-constraint} is the same.
\end{corollary}
\begin{proof}
	The results follows by applying Lemma \ref{lem:equality-of-inducing-sets} to the routing game where the cost functions $C_p^{\theta_{s}}(\cdot)$ are replaced with $\phi^u_s C_p^{\theta_{s}}(\cdot)$.
\end{proof}
Our first result considers the two-state case, and shows that there exists a set of distributions which induce flows in $\FF_{\info}$.
\begin{lemma}\longthmtitle{Distributions leading to $\FF_\info$} \label{lem:useful-WE-interval}
	Let $\PP$, $\Theta$, and $\CC$ be given, where $\Theta = \{\theta_1,\theta_2\}$ and ${\WW^{\theta_{1}} \neq \WW^{\theta_{2}}}$. Let $\FF_{\info}$ be as given in~\eqref{eq:F-inf}. There exist distributions $\xi \neq \eta$ with $\xi_1 > \eta_1$ such that for any ${\varphi_\mu := \mu \varphi^{\theta_1} + (1 - \mu)\xi}$, ${\mu \in [0,1]}$ we have $\WW^{\varphi_{\mu}} = \WW^{\theta_{1}}$, and for any ${\varphi_\lambda := \lambda \xi + (1 - \lambda)\eta}$, $\lambda \in (0,1)$ there exists $f^{\varphi_{\lambda}} \in \WW^{\varphi_\lambda}$ such that
	\begin{align}\label{eq:finf-inclusion}
		f^{\varphi_{\lambda}} \in \FF_{\info}.
	\end{align}
\end{lemma}
\begin{proof}
	First we aim to find the distribution ${\xi := (\xi_1,1 - \xi_1)}$. Pick any $f^{\theta_{1}} \in \WW^{\theta_{1}}$. From Lemma~\ref{lem:interval-of-constant-WE}, the set of distributions $\varphi \in \Delta_1^2$ with $f^{\theta_1} \in \WW^\varphi$ is convex and compact. That is, there exist a $c \in [0,1]$ such that ${f^{\theta_{1}} \notin \WW^\varphi}$ for all $\varphi \in \Delta_1^2$ with $\varphi_1 < c$ and $f^{\theta_{1}} \in \WW^\varphi$ for all $\varphi \in \Delta_1^2$ with $\varphi_1 \geq c$. In addition, by Lemma~\ref{lem:equality-and-intersection}, $f^{\theta_{1}} \in \WW^\varphi$ for some $\varphi$ if and only if $\WW^\varphi = \WW^{\theta_{1}}$. Combining these two facts and setting $\xi := (c,1-c)$ yields that: (a) $\WW^{\varphi_\mu} = \WW^{\theta_1}$ for all $\varphi_\mu = \mu \varphi^{\theta_1} + (1 - \mu)\xi$ and $\mu \in [0,1]$; and (b) for all $\varphi \in \Delta_1^2$ with $\varphi_1 < \xi_1$ and all $f^\varphi \in \WW^\varphi$, we have $f^\varphi \notin \WW^{\theta_1}$. The latter item (b) shows that $\xi_1 > 0$ which is essential for a $\eta$ distribution with $\eta_1 < \xi_1$ to exist. To see the reasoning for $\xi_1 > 0$, note that $\WW^{\theta_1} \not = \WW^{\theta_2}$ and by Lemma~\ref{lem:equality-and-intersection}, $\WW^{\theta_1} \cap \WW^{\theta_2} = \emptyset$. This statement will contradict if $\xi_1 = 0$ as then, $\WW^{\varphi_\mu} = \WW^{\theta_2}$ for $\mu=0$. The next step is to find the distribution $\eta := (\eta_1, 1-\eta_1)$. Let $f^\xi_{\mathrm{edge}}$ be the edge-flows associated to any $\xi$-WE. Pick any $p \in \RR^{\mathrm{use}}_{\eta}$ and by the second implication of Lemma~\ref{lem:used-set-properties}, $(f^\xi_{\mathrm{edge}})_k > 0$ for all $e_k \in p$. By continuity property of Lemma~\ref{lem:continuity-edge-flow}, there exist $\delta_\xi > 0$ such that $(f^\eta_{\mathrm{edge}})_k > 0$ for edge-flows associated to any $\eta$-WE where $\norm{\xi - \eta} < \delta_\xi$. This along with the second implication of Lemma~\ref{lem:used-set-properties} implies that $ \RR^{\mathrm{use}}_{\xi} \subseteq \RR^{\mathrm{use}}_{\eta}$ for all $\eta$ satisfying $\norm{\xi - \eta} < \delta_\xi$. This along with $\WW^\xi = \WW^{\theta_{1}}$ gives us
	\begin{align*}
		{\RR^{\mathrm{use}}_{\varphi^{\theta_{1}}} \subseteq \RR^{\mathrm{use}}_{\eta}},
	\end{align*}
	for all $\eta$ satisfying $\norm{\xi - \eta} < \delta_\xi$. From the first claim of Lemma~\ref{lem:used-set-properties}, there exists $f^{\eta} \in \WW^{\eta}$ such that $f_p^\eta > 0$ for all $p \in \RR^{\mathrm{use}}_{\varphi^{\theta_{1}}}$. Further, restricting our attention to $\eta$ with $\eta_1 < \xi_1$, we also know that $f^\eta \notin \WW^{\theta_{1}}$. In order to establish~\eqref{eq:finf-inclusion}, we now show that $f^\eta \in \FF_\info$. To this end, given above properties of $f^\eta$, all that remains to be shown is that setting $\delta_\xi$ small enough ensures $f^\eta \notin \WW^{\theta_2}$. For this, note that since $\WW^{\theta_{1}} \cap \WW^{\theta_{2}} = \emptyset$, we have $f_{\mathrm{edge}}^{\theta_{1}} \neq f_{\mathrm{edge}}^{\theta_{2}}$. Consequently, by Lemma~\ref{lem:continuity-edge-flow} and $\WW^\xi = \WW^{\theta_1}$, it follows that there exists $\delta_0 > 0$ such that $\norm{\xi - \eta} < \delta_0$ gives $f^{\eta} \notin \WW^{\theta_{2}}$. Thus, setting $\delta_\xi < \delta_0$ implies that for any $\eta \in \Delta_1^2$ with $\norm{\xi - \eta} < \delta_\xi$ and $\eta_1 < \xi_1$ there exists $f^{\eta} \in \WW^{\eta}$ such that $f^{\eta} \in \FF_{\info}$. Fixing $\eta_1$ as the infimum over all values for which $\norm{\xi - \eta} < \delta_\xi$ holds finishes the proof.
\end{proof}
Figure \ref{fig:4-road-2-state} can help us gain some intuition about the implications of Lemma \ref{lem:useful-WE-interval}. Under the given assumptions, the result divides the set $\Delta_1^2$ of all distributions into three convex regions. The first region is compact, and for any distribution inside of it the induced flows are contained in $\WW^{\theta_{1}}$. In Figure \ref{fig:4-road-2-state} we see that this region is the singleton set $\{(0,1)\}$. The second region is a convex and open set of distributions bordering the first region, for which the induced flows are in $\FF_\info$. In Figure \ref{fig:4-road-2-state} this would be all distributions between $\varphi_1 = 1$ and the first point where the flow on path 1 becomes zero. Note that any flow in this region is induced by a unique distriubion. The third region then contains all other distributions. Note that in this third region there are still flows that are uniquely associated to only one distribution. The next result shows that if for a given distribution $\varphi$ there exists a $\varphi$-WE $f^\varphi$ such that $f^\varphi \in \FF_{\info}$, then the constraints \eqref{eq:VI-cond-specific2} for \emph{any} $\varphi$-WE uniquely determine $\varphi$.
\begin{lemma}\longthmtitle{Informativity of flows in $\FF_\info$} \label{lem:Finf-identifies-distribution}
	Let $\PP$, ${\Theta = \{\theta_{1},\theta_{2}\}}$, $\CC$, and $\varphi \in \Delta_1^2$ be given. If there exists $f^\varphi \in \WW^\varphi$ satisfying $f^\varphi \in \FF_{\info}$, then $\varphi$ is the unique solution to \eqref{eq:VI-cond-specific2} for any $\varphi$-WE $\widehat{f}^\varphi \in \WW^\varphi$.\footnote{Here we replace $f^q$ in~\eqref{eq:VI-cond-specific2} with $\widehat{f}^\varphi$ and treat $q$ as a variable that can be solved for.}
\end{lemma}
\begin{proof}
	Let $f^{\varphi,1} \in \WW^\varphi$ be a flow such that $f^{\varphi,1}_p > 0$ for all $p \in  \RR^{\mathrm{use}}_\varphi$, which exists by Lemma~\ref{lem:used-set-properties}. By assumption, there exists a WE $f^{\varphi,2} \in \WW^{\varphi}$ satisfying $f^{\varphi,2} \in \FF_{\info}$. Let $\widehat{f}^\varphi \in \WW^\varphi$ be defined as $\widehat{f}^\varphi:=\lambda_1 f^{\varphi,1}  + \lambda_2 f^{\varphi,2}$ for some $\lambda_1, \lambda_2 > 0$ with $\lambda_1 + \lambda_2 = 1$. Note that $\widehat{f}_p^\varphi > 0$ for all $p \in \RR^{\mathrm{use}}_\varphi \cup \RR^{\mathrm{use}}_{\theta_1}$ and one can select $\lambda_1$ and $\lambda_2$ additionally to ensure $\widehat{f}^\varphi \in \FF_{\info}$. Picking such constants and noting the definition of $\FF_{\info}$, we have ${\widehat{f}^\varphi \notin \WW^{\theta_{1}}}$, meaning that $\widehat{f}^{\varphi}$ is not a $\varphi^{\theta_{1}}$-WE. We will next show that $\varphi$ is the unique solution to \eqref{eq:VI-cond-specific2} where $f^q$ is replaced with $\widehat{f}^\varphi$ and $q$ is treated as a variable to be solved for. Note that since $\widehat{f}^{\varphi}$ is not a $\varphi^{\theta_{1}}$-WE, there exist paths $p,r \in \PP$ such that $\widehat{f}^\varphi_p > 0$ and
	\begin{equation} \label{eq:useful-constraint-1}
		C^{\theta_{1}}_p(\widehat{f}^\varphi) > C^{\theta_{1}}_r(\widehat{f}^\varphi).
	\end{equation}
	Consider two cases: (a) $\widehat{f}^\varphi_r > 0$ and (b) $\widehat{f}^\varphi_r = 0$. For case (a), from~\eqref{eq:VI-cond-specific2}, we obtain an equality constraint of the form
	\begin{align*}
		\left(\begin{array}{cc} C_p^{\theta_1}(\widehat{f}^\varphi) - C_r^{\theta_1}(\widehat{f}^\varphi)& C_p^{\theta_2}(\widehat{f}^\varphi) - C_r^{\theta_2}(\widehat{f}^\varphi)
		\end{array}\right) \varphi = 0.
	\end{align*}
	Since $C^{\theta_{1}}_p(\widehat{f}^\varphi) \neq C^{\theta_{1}}_r(\widehat{f}^\varphi)$ this constraint along with ${\varphi_1 + \varphi_2 = 1}$ gives us two linearly independent equality constraints on $\varphi$. Since $\varphi \in \real^2$ this implies that $\varphi$ is the only distribution that satisfies the constraints in~\eqref{eq:VI-cond-specific2}. We next show that case (b), with $\widehat{f}^\varphi_r = 0$, does not occur. To be precise, we claim that for $\widehat{f}^\varphi \notin \WW^{\theta_1}$, there exists at least one pair of paths $p,r \in \PP$ satisfying~\eqref{eq:useful-constraint-1} where both $\widehat{f}_p > 0$ and $\widehat{f}_r > 0$. To show this, we proceed with a contradiction argument. Assume there does not exist such a pair of paths. This implies two things: 1) $C^{\theta_{1}}_p(\widehat{f}^\varphi) = C^{\theta_{1}}_r(\widehat{f}^\varphi)$ for all $p,r \in \PP$ such that $\widehat{f}_p^\varphi > 0$ and $\widehat{f}_r^\varphi > 0$; 2) if $C^{\theta_{1}}_p(\widehat{f}^\varphi) > C^{\theta_{1}}_r(\widehat{f}^\varphi)$ for some $p$ such that $\widehat{f}_p^\varphi > 0$ this implies $\widehat{f}_r^\varphi = 0$. Now consider the graph $\widecheck{\GG} = (\widecheck{\VV},\widecheck{\EE})$ with $\widecheck{\VV} = \VV$ and $\widecheck{\EE} \subseteq \EE$ such that $e_k \in \widecheck{\EE}$ if and only if $(\widehat{f}^\varphi_{\mathrm{edge}})_k > 0$, where $\widehat{f}^\varphi_{\mathrm{edge}}$ is the vector of edge-flows. To all edges in $\widecheck{\EE}$ associate the same state-dependent cost functions as in the original network, and consider the same set of states $\Theta$. This defines a new routing game over the network $\widecheck{\GG}$. Note that for any $p \in \PP$ such that $\widehat{f}^\varphi_p > 0$, we have by Lemma~\ref{lem:used-set-properties} that $(\widehat{f}^\varphi_{\mathrm{edge}})_k > 0$ for all $e_k \in p$ Thus, when $\widehat{f}^\varphi_p > 0$ and $e_k \in p$, then $e_k \in \widecheck{\EE}$. Therefore, $\widehat{f}_p^\varphi > 0$ implies $p \in \widecheck{\PP}$. Thus, we can define a feasible flow for the modified game by setting $\widecheck{f}^\varphi_p := \widehat{f}^\varphi_p$ for all $p$ such that $\widehat{f}^\varphi_p > 0$.  Since cost functions over the used edges have not changed, if for two paths $p$ and $r$ we have $C^{\theta_{1}}_p(\widehat{f}^\varphi) = C^{\theta_{1}}_r(\widehat{f}^\varphi)$, then $\widecheck{C}^{\theta_{1}}_p(\widecheck{f}^\varphi) = \widecheck{C}^{\theta_{1}}_r(\widecheck{f}^\varphi)$. Now recall that if there was a path $r$ in the original game such that $C^{\theta_{1}}_p(\widehat{f}^\varphi) > C^{\theta_{1}}_r(\widehat{f}^\varphi)$ for some $p$ satisfying $\widehat{f}^\varphi_p > 0$, then by 2) we have $\widehat{f}^\varphi_r = 0$. Since $\widehat{f}^\varphi_p > 0$ for all $p \in \RR^{\mathrm{use}}_\varphi$ by construction, we deduce that $f_r = 0$ for all $f \in \WW^\varphi$. Using the second implication of Lemma~\ref{lem:used-set-properties} we see that there exists some edge $e_k \in r$ such that $(\widehat{f}^\varphi_{\mathrm{edge}})_k = 0$. Therefore, the edge $e_k$ has been removed in the modified game, and it follows that the path $r$ is not present in the modified game. In conclusion, we have  $\widecheck{C}^{\theta_{1}}_p(\widecheck{f}^\varphi) = \widecheck{C}^{\theta_{1}}_r(\widecheck{f}^\varphi)$ for all $p,r \in \widecheck{\PP}$ which implies that $\widecheck{f}^\varphi$ is a $\varphi^{\theta_{1}}$-WE for the modified game. Now consider any flow $f^{\theta_{1}} \in \WW^{\theta_{1}}$. If $f^{\theta_{1}}_p > 0$, then $p \in \RR^{\mathrm{use}}_{\varphi^{\theta_{1}}}$, which implies $\widehat{f}^\varphi_p > 0$. Repeating the above arguments then shows that $p \in \widecheck{\PP}$. Thus, we can define a feasible flow for the modified game by setting $\widecheck{f}^{\theta_{1}}_p = f^{\theta_{1}}_p$. Similar to before we have that since $f^{\theta_{1}}$ is a $\varphi^{\theta_{1}}$-WE of the original game this implies that $\widecheck{f}^{\theta_{1}}$ is a $\varphi^{\theta_{1}}$-WE for the modified game. However, since $\widehat{f}^\varphi \in \FF_{\info}$, we have $\widehat{f}^\varphi \notin \WW^{\theta_{1}}$ which implies that $\widehat{f}^\varphi_{\mathrm{edge}} \neq \widecheck{f}^{\theta_{1}}_{\mathrm{edge}}$, whereas $\widehat{f}^\varphi_{\mathrm{edge}} = \widecheck{f}^\varphi_{\mathrm{edge}}$. This means we obtain two $\varphi^{\theta_{1}}$-WE, namely $\widecheck{f}^\varphi$ and $\widecheck{f}^{\theta_1}$ for the modified game with unequal edge-flows. This contradicts the uniqueness of edge-flow under $\varphi^{\theta_{1}}$-WE. Thus we arrive at a contradiction. Therefore there do exist $p,r \in \PP$ such that $\widehat{f}_p^\varphi > 0$, $\widehat{f}_r^\varphi > 0$, and~\eqref{eq:useful-constraint-1} holds. Therefore, $\varphi$ is uniquely determined by the constraints in~\eqref{eq:VI-cond-specific2}. From Lemma~\ref{lem:equality-of-inducing-sets}, we have that for any $\widehat{f}^\varphi \in \WW^{\varphi}$ the set of priors satisfying the constraints imposed by \eqref{eq:VI-cond-specific2} is the same, which then concludes the proof.
\end{proof}
We next present the main result of this section. In it we make use of Lemma's \ref{lem:useful-WE-interval} and \ref{lem:Finf-identifies-distribution} to design a signalling scheme for which all but one of the signals give an equality constraint on the prior, showing that there always exists a signalling scheme using $m$ messages that is $q$-identifying.
\begin{proposition}\longthmtitle{Existence of $q$-identifying signalling scheme} \label{prop:existence-signalling-scheme}
	Let $\PP$, $\Theta$, $\CC$, and $q$ be given, and assume that ${\WW^{\theta_{1}} \neq \WW^{\theta_{2}}}$.\footnote{By relabeling the states we can see that this assumption is equivalent to assuming existence of two states $\theta_k, \theta_{\ell} \in \Theta$ such that ${\WW^{\theta_{k}} \neq \WW^{\theta_{\ell}}}$.} 
	Then, there exists a signalling scheme ${\Phi \in \colsto{m}{m}}$ of $m$ messages that is $q$-identifying.
\end{proposition}
\begin{proof}
	Our proof will be constructive. Recall the matrix notation of the signalling scheme, that is, $\Phi = (\phi_s^u)_{u,s \in [m]}$, where $\phi_s^u$ is the $(u,s)$-th entry of the matrix and denotes the probability of sending signal $\zeta^u$ under the state $\theta_s$. We will proceed row-by-row starting from the second row of $\Phi$. 
	
	\emph{Step 1: Constructing the second row:} Set $\phi^2_s = 0$ for all $s \in [m] \setminus \{1,2\}$. Using \eqref{eq:Bayes-rule} we obtain the posterior distribution under the message $\zeta^2$ as 
	\begin{equation} \label{eq:existence-prop-posterior}
		\begin{split}
			\widetilde{q}^{\zeta^2}_{1} &= \frac{\phi^2_1 q_1}{\phi^2_1 q_1 + \phi_{2}^2q_2}, \quad
			\widetilde{q}^{\zeta^2}_2 = \frac{\phi^2_2q_2}{\phi^2_1 q_1 + \phi^2_2 q_2},
			\\
			\widetilde{q}^{\zeta^2}_s &= 0, \quad \text{for all } s \in [m] \setminus \{1,2\}.
		\end{split}
	\end{equation}
	Since $q_1$ and $q_2$ are non-zero by assumption, one can tune $\phi^2_1$ and $\phi^2_2$ to induce any posterior $\widetilde{q}^{\zeta^2}$ satisfying $0 < \widetilde{q}^{\zeta^2}_1 < 1$ and $\widetilde{q}^{\zeta^2}_1 = 1 - \widetilde{q}^{\zeta^2}_2$. By construction we then have $\phi^2_1, \phi^2_2 > 0$. In the following, we will outline the procedure for tuning these parameters such that the flow induced by signal $\zeta^2$ results in an equality constraint for the prior $q$.
	
	Observe that when considering the signal $\zeta^2$, we have simplified the situation by removing the influence from all but the first two states on the posterior (by setting $\phi^2_s = 0$ for $s=3,4,\dots$). That is, we have effectively reduced the analysis to a two state case, as analyzed in Lemma's~\ref{lem:useful-WE-interval} and \ref{lem:Finf-identifies-distribution}. Consequently we can appeal to Lemma \ref{lem:useful-WE-interval} to conclude that there exists a posterior
	$\widetilde{q}$, with $\widetilde{q}_1 \in (0,1)$, $\widetilde{q}_2 = 1 - \widetilde{q}_1$, and $\widetilde{q}_s = 0$ for all $s = [m] \setminus \{1,2\}$, such that there exists a $\widetilde{q}$-WE $f^{\widetilde{q}}$ satisfying $f^{\widetilde{q}}  \in \FF_{\info}$, where $\FF_{\info}$ is given in~\eqref{eq:F-inf}. From Lemma \ref{lem:Finf-identifies-distribution} we know that if there exists a $\widetilde{q}$-WE that lies in $\FF_{\info}$, then the constraints in \eqref{eq:VI-cond-specific2} generated by any $\widetilde{q}$-WE allow for unique identification of $\widetilde{q}$.\footnote{That is, it allows us to determine $\widetilde{q}_1$, and $\widetilde{q}_2$. Since we already know that $\widetilde{q}_s = 0$ for all $s \notin \{1,2\}$ this fully identifies $\widetilde{q}$.} Now pick $\phi_1^2$ and $\phi_2^2$ such that the posterior $\widetilde{q}$ with $f^{\widetilde{q}} \in \FF_\info$ is induced under the signal $\zeta^2$. Consequently, substituting $\widetilde{q} = \widetilde{q}^{\zeta^2}$ into \eqref{eq:existence-prop-posterior} then gives the constraint
	\begin{equation}\label{eq:equlity-obt}
		\frac{\phi_{2}^2}{\phi_1^2(1 - \widetilde{q}_1)}q_2 = q_1.
	\end{equation}
	This constraint is well-posed and non-trivial since $\phi_2^2$ and $\phi_1^2$ are non-zero by design, and as noted $\widetilde{q}_1 =  \widetilde{q}^{\zeta^2}_1 < 1$. Thus, by tuning the values $\phi_1^2$ and $\phi_2^2$, we are able to find an equality constraint~\eqref{eq:equlity-obt} on the prior. 
	
	\emph{Step 2: Constructing rows $3$ through $m$:} For row $s \notin \{1,2\}$ it follows from Lemma~\ref{lem:equality-and-intersection} and $\WW^{\theta_{1}} \neq \WW^{\theta_{2}}$ that we have either $\WW^{\theta_{s}} = \WW^{\theta_{1}}$, in which case $\WW^{\theta_{s}} \neq \WW^{\theta_{2}}$, or we have $\WW^{\theta_{s}} \neq \WW^{\theta_{1}}$. In other words, there exists a state $\theta_\ell$ with $\ell < s$ such that $\WW^{\theta_{s}} \neq \WW^{\theta_{\ell}}$. By setting $\phi^{s}_{s'} = 0$ for all $s' \notin \{s,\ell\}$ we can, similar to before, induce any posterior $\widetilde{q} = \widetilde{q}^{\zeta^s}$ such that $\widetilde{q}_{s'} = 0$ for all $s' \notin \{s,\ell\}$ and so, $\widetilde{q}_s = 1 - \widetilde{q}_\ell$. We can then repeat the previous arguments to show that by tuning $\phi^s_s$ and $\phi^s_\ell$ we can obtain a well-posed, non-trivial equality constraint on $q$ of the form
	\begin{equation*}
		\frac{\phi_{s}^s}{\phi_\ell^s(1 - \widetilde{q}_\ell)}q_s = q_\ell.
	\end{equation*}
	This equality constraint is necessarily linearly independent from the other equality constraints obtained in this manner. To see this note that the constraint generated by row $s$ involves $q_s$, while the set of constraints generated by the rows $s' < s$ do not involve $q_s$ by construction. Thus, in this manner we obtain $m-1$ linearly independent equality constraints $q$.
	
	\emph{Step 3: Constructing the first row:} Once we have constructed the rows $2$ through $m$ of $\Phi$, we select the elements of the first row such that each column of $\Phi$ sums to one. This is always possible and a short procedure is given in Algorithm~\ref{ag:DQSS}.
	
	Finally, note that with the $m-1$ linearly independent equality constraints and the additional independent constraint $\mathbb{1}^\top q = 1$ derived from the condition that $q \in \Delta_1^m$ lies in the simplex, we obtain $m$ constraints that uniquely identify $q$.
\end{proof}
\begin{remark}\longthmtitle{Drawbacks of signals limited to two states} \label{re:signalling-scheme-limitations}
	\rm{In the proof of Proposition \ref{prop:existence-signalling-scheme} we make use of a specific kind of signalling scheme in which each signal except the first has a positive chance of being send only in two states, and the first signal is used to ensure that the signalling scheme satisfies all the required constraints. Mathematically, the signalling scheme belongs to the set
		\begin{equation} \label{eq:phi_set}
			\begin{split}
				&\mathcal{S}_{\mathrm{sig}}  \! := \! \setdefb{\Phi \!  \in \! \colsto{m}{m} \! }{ \! \forall \hspace{2 pt} s \! \in \! [m] \! \setminus \! \{1\}, \enskip  \exists \ell \! < \! s 
					\\ 
					& \hspace{10 pt}\text{such that }  \WW^{\theta_{s}}  \neq  \WW^{\theta_{\ell}}, \text{and } \phi^s_{s'} = 0 \Leftrightarrow s' \not \in \{s,\ell\} }.
			\end{split}
		\end{equation}
		Such a scheme is used because for each signal, as mentioned, the situation is effectively reduced to a two state case, allowing for simpler analysis. However, such a scheme is limited in that it can derive at most one equality constraint from a signal. If a signal can be send in more than two states, more information may be gained. Analysis however becomes more difficult, since it is not clear if and how the result of Lemma \ref{lem:useful-WE-interval} can be generalized to a case involving more than two states.}
	\oprocend
\end{remark}
\subsection{Designing the signalling scheme} \label{ssec:design}
With existence of a $q$-identifying signalling scheme guaranteed under mild conditions, the next step would be to give guidelines for how such a scheme can be designed. For this purpose we provide Algorithm~\ref{ag:DQSS},  which using observations of the flow under various signals, updates a signalling scheme until it is $q$-identifying. The algorithm uses signalling schemes in the set~\eqref{eq:phi_set}, and requires the assumption of Proposition~\ref{prop:existence-signalling-scheme} that the sets $\WW^{\theta_{1}}$ and $\WW^{\theta_{2}}$ are not equal.
\begin{quote}
	\emph{[Informal description of Algorithm \ref{ag:DQSS}]:} The procedure starts with an initial $\Phi(0)$ of the form \eqref{eq:phi_set}, such that for each signal $s \in [m] \setminus \{1\}$ exactly two elements in the $s$-th row of $\Phi(0)$ are non-zero. One of these elements is $\phi_s^s$ and the other is denoted $\phi^s_{\ell(s)}$ (cf. Line~\ref{step1}). At each iteration $N$, and for each row $s \in \II$, we check whether the flow $\widetilde{f}^{\zeta^s}$ observed under the signalling scheme $\Phi(N)$ when sending signal $\zeta^s$ results in an equality constraint on $q$ (cf. Line~\ref{step-check}). If it does, then row $s$ of $\Phi(N)$ is not updated in the for-loop and the ratio between $\phi_s^s$ and $\phi_{\ell(s)}^k$ remains the same for all subsequent iterations (cf. Lines~\ref{step-keep-same} and \ref{step-removal}). If not, then we consider two cases. In the first case, the flow $\widetilde{f}^{\zeta^s} \in \WW^{\theta_{s}}$ and the values $\phi_s^s$ and $\phi^s_{\ell(s)}$ are updated so as to increase the ratio $\phi^s_{\ell(s)}/\phi_s^s$ in signalling scheme $\Phi(N+1)$. In this way, the posterior induced by signal $s$ in the next iteration will assign less probability to state $\theta_s$. This increase in ratio is achieved in Lines~\ref{step-inc-1} through~\ref{step-inc-3}. In the second case, $\widetilde{f}^{\zeta^s} \notin \WW^{\theta_{s}}$ and we decrease the ratio $\phi^s_{\ell(s)}/\phi_s^s$ in Line~\ref{step-dec-1}. After modifying rows in this way, the signalling scheme $\Phi(N+1)$ is updated in Lines~\ref{step:sum-to-11}-\ref{step:seum-to-12} so as to ensure that each column sums to unity while preserving the ratios $\phi^s_{\ell(s)}/\phi_s^s$.
\end{quote}
\begin{algorithm}[!htbp]
	\SetAlgoLined
	\DontPrintSemicolon
	\SetKwInOut{ini}{Initialize}
	\SetKwInOut{giv}{Data}
	\ini{An index set $\mathcal{I} = \{2,3, \cdots, m\}$, counter $N = 0$,  lower and upper bounds $\lowra_s (0) = 0$, $\upra_s (0) = \infty$ for all $s \in \II$, a signalling scheme $\Phi(0) \in \mathcal{S}_{\mathrm{sig}}$ using~\eqref{eq:phi_set}}
	For all $s \in \II$ set $\ell(s) \neq s$ such that $\phi^s_{\ell(s)}(0) \neq 0$ \; \label{step1}
	Compute $f^{\theta_s}$ for all $s \in [m]$ by solving $\VI(\HH,C^{\theta_s})$\;
	\While{$\mathcal{I} \neq \emptyset$}{
		Set $\phi^j_i(N+1) \gets \phi^j_i (N)$ for all $i \in [m]$ and $j \in [m]\setminus \II$ \; \label{step-keep-same}
		\For{$s \in \mathcal{I}$}{
			Obtain $\widetilde{f}^{\zeta^s}$ under scheme $\Phi(N)$ and signal $\zeta^s$ \;
			Check if $\widetilde{q}^{\zeta^s}$ is uniquely determined by \eqref{eq:VI-cond-specific2}\footnotemark \;
			\uIf{$\widetilde{q}^{\zeta^s}$ \rm{is uniquely identified}
				\label{step-check} }{Set $\mathcal{I} \gets \mathcal{I} \setminus \{s\}$\; \label{step-removal}}
			\uElseIf{$\widetilde{f}^{\zeta^s} \in \WW^{\theta_{s}}$}{\setlength{\abovedisplayskip}{-12pt}
				\setlength{\belowdisplayskip}{-8pt}
				\begin{flalign*}
					\! \! \text{Set }\lowra_s(N+1) &\gets \tfrac{\phi^s_{\ell(s)} (N)}{\phi^s_s(N)},	&&\\
					\upra_s(N+1) &\gets \upra_s(N), \text{ and}	&&\\
					\phi^s_s(N+1) &\gets \phi^s_s(N)&&
				\end{flalign*} \label{step-inc-1}\;
				\eIf{$\upra_s(N+1) = \infty$}{Set $\phi^s_{\ell(s)}(N+1) \gets 2 \phi^s_{\ell(s)} (N)$ \label{step-inc-2} }{Set $\phi^s_{\ell(s)}(N+1) \gets \frac{1}{2}\big(\lowra_s(N+1) + \upra_s(N+1)\big) \phi^s_s(N)$\label{step-inc-3}}}
			\Else{\setlength{\abovedisplayskip}{-14pt}
				\setlength{\belowdisplayskip}{-10pt}
				\begin{flalign*}
					\! \! \text{Set }\upra_s(N+1) &\gets \tfrac{\phi^s_{\ell(s)} (N)}{\phi^s_s(N)},	&&\\
					\lowra_s(N+1) &\gets \lowra_s(N), \text{ and}	&&\\
					\phi^s_s(N+1) &\gets \phi^s_s(N)&&	\\
					\phi^s_{\ell(s)}(N+1) & \gets&& \\
					\tfrac{1}{2}\big(\lowra_s(N+1) &+ \upra_s(N+1)\big) \phi^s_s(N)&&
				\end{flalign*}  \label{step-dec-1}\;
			}
		}{Set $a = \max_{i \in [m]} \sum_{j \in [m] \setminus \{1\}} \phi_i^j(N+1)$. \label{step:sum-to-11}\;}
		{Set $\Phi(N+1) \gets \frac{1}{a}\Phi(N+1)$}\;
		\For{$s \in [m]$}
		{Set $\phi^1_s(N+1) \gets 1 - \sum_{j \in [m] \setminus \{1\}} \phi_s^j(N+1)$ \label{step:seum-to-12} \;}
		{Set $N \gets N + 1$}\;
	}
	\caption{Find $q$-identifying Signalling Scheme}
	\label{ag:DQSS}
\end{algorithm}
The\footnotetext{Here $q$ and $f^q$ in \eqref{eq:VI-cond-specific2} are replaced with $\widetilde{q}$ and $f^{\widetilde{q}}$ respectively.}  above procedure identifies the right signalling scheme, and can also determine the prior, since the obtained constraints define it uniquely.  Next we establish the correctness of Algorithm~\ref{ag:DQSS}.
\begin{proposition}\longthmtitle{Convergence of Algorithm~\ref{ag:DQSS}}
	Let $\PP$, $\Theta$, $\CC$, and $q$ be given, and assume that $\WW^{\theta_1} \neq \WW^{\theta_2}$. Then, Algorithm~\ref{ag:DQSS} terminates in a finite number of iterations $N_f$, and the resulting signalling scheme $\Phi(N_f)$ is $q$-identifying.
\end{proposition}
\begin{proof}
	For a signal $s \in [m] \setminus \{1\}$, we look at the properties of $\widetilde{f}^{\zeta^s}$, $\lowra_s(N)$, $\upra_s(N)$, and $r_s(N) := \frac{\phi^s_{\ell(s)}}{\phi^s_s}$ as the algorithm iterates. We first show that $\lowra_s(N) \leq \upra_s(N)$ for all $N$, which holds by definition for the initial iterate. We suppress the argument $N$ in the  following few statements for the sake of convenience. Observe that the signalling scheme maintains the same sparsity pattern, of the form~\eqref{eq:phi_set}, in all iterations. That is, $\phi^s_i = 0$ for all $i \notin \{\ell(s),s\}$ and all iterations. This effectively reduces the analysis to that of a two-state situation, meaning that the posterior under signal $\zeta^s$, denoted $\widetilde{q} = \widetilde{q}^{\zeta^s}$, satisfies $\widetilde{q}_i = 0$ for all $i \notin \{\ell(s),s\}$ and any choice of $\phi^s_{\ell(s)}$, $\phi^s_s$. From Lemma~\ref{lem:useful-WE-interval}, there exist constants $a_s \in (0,1]$ and $c_s \in (0,1)$ with $a_s > c_s$ such that
	\begin{equation}\label{eq:a-prop}
		\begin{aligned}
			\widetilde{q}_s \geq a_s &\Rightarrow \WW^{\widetilde{q}} \cap \WW^{\theta_{s}} = \emptyset,
			\\
			a_s > \widetilde{q}_s > c_s &\Rightarrow  \exists f^{\widetilde{q}} \in \WW^{\widetilde{q}} \text{ such that } f^{\widetilde{q}} \in \FF_{\info}.\text{\footnotemark}
		\end{aligned}
	\end{equation}	
	\footnotetext{Here, instead of considering all $f \in \HH$ we only consider flows $f$ for which there exists a ${\varphi \in \setdef{\xi \in \Delta_1^m}{\xi_{s'} = 0, \text{ for } s' \notin \{\ell(s),s\}} \setminus \{\varphi^{\theta_{s}}\}_{s\in [m]}}$ such that $f \in \WW^{\varphi}$ in the definition of $\FF_{\info}$.} From~\eqref{eq:Bayes-rule}, we have
	\begin{align*}
		\widetilde{q}_s &= \frac{\phi^s_s q_s}{\phi^s_s q_s + \phi^s_{\ell(s)} q_{\ell(s)}}	= \frac{ q_s}{ q_s + \frac{\phi^s_{\ell(s)}}{\phi^s_s} q_{\ell(s)}}.
	\end{align*}
	Note that the influence of $\Phi$ on $\widetilde{q}_s$ is completely determined by the ratio $r_s = \frac{\phi^s_{\ell(s)}}{\phi^s_s}$ and that $\widetilde{q}_s$ is monotonically decreasing in $r_s$ with $\lim_{r_s \rightarrow \infty} \widetilde{q}_s = 0$ and $\lim_{r_s \rightarrow 0} \widetilde{q}_s = 1$. Thus, given \eqref{eq:a-prop}, and the relationship between $\widetilde{q}_s$ and $r_s$, we deduce that there exist constants $b_s \ge 0$ and $d_s > b_s$ such that
	\begin{equation}\label{eq:r-prop}
		\begin{aligned}
			r_s \le  b_s &\Rightarrow f^{\widetilde{q}} \in \WW^{\theta_{s}},
			\\
			d_s > r_s >  b_s &\Rightarrow f^{\widetilde{q}} \in \FF_{\info}.
		\end{aligned}
	\end{equation}
	With this in mind, we now analyze the evolution of $\lowra_s$ and $\upra_s$. Note that $\lowra_s$ is only changed in line~\ref{step-inc-1} of the algorithm. Here we set $\lowra_s(N) = r_s(N)$ whenever ${f^{\widetilde{q}} \in \WW^{\theta_{s}}}$ and thus from~\eqref{eq:r-prop}, $\lowra_s(N) \leq b_k$ for all $N$. Similarly, ${\upra_s(N) = r_s(N)}$ whenever $f^{\widetilde{q}} \notin \WW^{\theta_{s}}$ and $\widetilde{q}^{\zeta^s}$ is not uniquely identified. As shown in the proof of Proposition~\ref{prop:existence-signalling-scheme}, whenever $f^{\widetilde{q}} \in \FF_{\info}$ we obtain an informative equality constraint. From \eqref{eq:r-prop} we then conclude that $\upra_s(N) \geq d_s$. We now have $\lowra_s(N) < \upra_s(N)$ for all $N$. In fact, we have ${(b_s,d_s) \subseteq \big(\lowra_s(N), \upra_s(N)\big)}$ for all $N$. We also note that since $d_s > b_s$ we have $d_s > 0$ and $b_s < \infty$. Now we look at the evolution of $r_s(N)$. We will show that $r_s(N) \in (b_s,d_s)$ for some finite $N$ and at that iteration, we obtain an informative equality constaint corresponding to signal $s$. This in turn shows termination of the algorithm in finite number of iterations. Consider three cases: (a) ${d_s = \infty}$; (b) $b_s = 0$ and $d_s < \infty$; and (c) otherwise. In case (a), at any $N$, we have either $r_s(N) \in (b_s,d_s)$ and we find an informative equality constraint, or $r_s(N) \leq b_s$, implying $f^{\widetilde{q}} = f^{\theta_s}$. In the latter case, $r_s(N+1) = 2 r_s(N)$. Thus, there exists some $\bar{N}$ such that $r_s(\bar{N}) > b_s$, implying $r_s(\bar{N}) \in (b_s,d_s)$. Similarly, in case (b), we have either $r_s(N) \in (b_s,d_s)$ or $r_s(N) \geq d_s$. In the latter case, $r_s(N)$ is halved for the next iteration and so in finite number of steps $r_s$ reaches $(b_s,d_s)$. In case (c), the arguments for case (a) and (b) can be repeated to show that there exists $\bar{N}$ such that $\lowra_s(\bar{N}) > 0$ and $\upra_s(\bar{N}) < \infty$. Looking at the algorithm, we see that $\upra_s(N) - \lowra_s(N)$ is halved in every subsequent iteration $N \ge \bar{N}$ . Since $r_s(N)$ always lies in the interval $\big(\lowra_s(N),\upra_s(N)\big)$, it then reaches the set $(b_s,d_s)$ in a finite number of iterations yielding an informative equality constraint. Following these facts, we conclude that an informative equality constraint is found in a finite number of iterations for each signal which completes the proof.
\end{proof}
\begin{remark}\longthmtitle{Practical considerations of implementing Algorithm~\ref{ag:DQSS}}\label{re:implement}
	{\rm The purpose of Algorithm \ref{ag:DQSS} is to demonstrate how insights from Proposition \ref{prop:existence-signalling-scheme} can be applied. It gives a methodical approach for constructing a $q$-identifying signalling scheme. However, it has several drawbacks worth noting:
		
		1) First we note that the TIS can only send one signal at any instance of the game, and does not have free choice of which signal to send, since after observing the state, the probability of a signal being sent is fixed by the current signalling scheme. Therefore, in practice, the TIS cannot send all signals in an ordered manner at each iteration of the algorithm and then update $\Phi$. Instead, it would be best to update a row of $\Phi$ after each instance of a game when the used signal does not induce a useful equality constraint. We have presented the algorithm in its current form, rather than the practically implementable one, to simplify the exposition of the main idea.
		
		2) When additional information on the prior is available, such as a lower bound $q_s \geq \epsilon > 0$ which holds for all ${s \in [m]}$, it may be possible to determine in advance which signalling scheme will supply informative constraints on the prior. For instance looking at Figure \ref{fig:2-road-2-state}, we see that whenever $p_1 \in (0.133,0.8)$ the result is a WE belonging uniquely to the associated distribution. If we then have, for instance, $p_1,p_2 \geq 0.25$ it follows that for this example an uninformative scheme (with $\phi^u_s = 0.5$ for all $u,s \in [2]$) is $q$-identifying.
		
		3) As mentioned in Remark \ref{re:signalling-scheme-limitations}, it may be beneficial to allow a signal to be send in more then two states, in order to obtain multiple equality constraints from a single signal. This may significantly reduce the number of iterations required to identify the prior, especially in combination with the above mentioned possibility of using additional knowledge about the prior to determine a signalling scheme in advance that necessarily provides informative constraints.
		
		4) Finally, we note that in this paper we have only considered the question of identifying the prior. In practice, the social cost incurred during the identification process is also important. For instance, once a signal $\zeta^u$ has resulted in an equality constraint on $q_\ell$ and $q_s$, that specific signal is no longer required for identification and can be modified with the aim of minimizing the social cost. However, the comparison between benefits of obtaining a better estimate of the prior and optimizing with respect to the current estimate is more involved and left for future work. \oprocend 
	}
\end{remark} \vspace{10 pt}
\begin{example}\longthmtitle{Application of Algorithm~\ref{ag:DQSS} in $4$-path $2$-state case}
	{\rm To shed light on the conclusions of Proposition \ref{prop:existence-signalling-scheme} and the workings of Algorithm \ref{ag:DQSS}, we revisit the 4-path, 2-state case in Example \ref{ex:constant-WE}. Setting $q = (0.5,0.5)^\top$, and using the initial signalling scheme
		\begin{equation*}
			\Phi(0) = 
			\left(
			\begin{array}{cc}
				0.5	&0.5	\\
				0.5	&0.5
			\end{array}
			\right),
		\end{equation*}
		we go through the steps of Algorithm \ref{ag:DQSS} to find a $q$-identifying signalling scheme. From \eqref{eq:4,2-example-cost} we derive $f^{\theta_1} = (0.8,0.2,0,0)^\top$ and $f^{\theta_2} = (0,0,0,1)^\top$.
		Using $\Phi(0)$ as a signalling scheme, \eqref{eq:Bayes-rule} gives us $\widetilde{q}^{\zeta^2} = \widetilde{q} = (0.5,0.5)^\top$. We can then use the functions $C^{\widetilde{q}}_p(f)$ and \eqref{eq:def-of-WE-constraint} to find that $f^{\widetilde{q}} = (0,\frac{5}{9},\frac{4}{9},0)^\top$ is the flow observed after sending signal $\zeta^2$. Even though two paths carry positive flow, the resulting constraint is trivial, since $C^{\theta_2}_2(\frac{5}{9}) - C^{\theta_2}_3(\frac{4}{9}) = 0$. In Figure \ref{fig:4-road-2-state} this can also be observed by noting that $\widetilde{q} = (0.5,0.5)^\top$ is in a region of distributions where the flow remains constant. We do have $f^{\widetilde{q}} \neq f^{\theta_2}$ which means that we will update $\Phi(0)$ according to Line~\ref{step-dec-1}. Setting the values as prescribed there, we get $\phi^2_1(1) = 0.25$, $\phi^2_2(1) = 0.5$. In Lines~\ref{step:sum-to-11}-\ref{step:seum-to-12} we then update the first row to ensure that all columns of $\Phi(1)$ sum to one, and thus we arrive at
		\begin{equation*}
			\Phi(1) = 
			\left(
			\begin{array}{cc}
				0.75	&0.5	\\
				0.25	&0.5
			\end{array}
			\right).
		\end{equation*}
		Using the new signalling scheme we find ${\widetilde{q}^{\zeta^2} = \widetilde{q} = (\frac{1}{3}, \frac{2}{3})^\top}$, resulting in $f^{\widetilde{q}} = (0,\frac{32}{68},\frac{23}{68},\frac{13}{68})^\top$. Substituting $\phi^2_1(1)$, $\phi^2_2(1)$ and $f^{\widetilde{q}}$ into \eqref{eq:posterior-WE-constraint}, where we set $p = 2$, $r = 4$, we get
		\begin{align*}
			\left(
			\begin{array}{c}
				\frac{1}{4}(\frac{32}{68}\frac{1}{2} + 1.7 - \frac{2}{5}\frac{13}{68} - 3.5) \vspace{2 pt}	\\
				\frac{1}{2}(\frac{32}{68}\frac{1}{2} + 1.7 - \frac{3}{5}\frac{13}{68} - 1)
			\end{array}
			\right)^\top q = 0.
		\end{align*}
		Solving this we find $q_1 = q_2$. Taken together with ${q_1 + q_2 = 1}$ this implies $q = (0.5,0.5)^\top$. Thus, the $q$-identifying scheme exists and is obtained in one iteration of the algorithm.} \oprocend
\end{example}
\section{Multiple priors and robust identification} \label{sec:relaxed-prior}
Here, we discuss possible generalizations of our setup that can bring it closer to real-life implementation. First we discuss the case where the population does not have a common prior and later we show  how the signalling schemes that we obtain can identify other priors.
\subsection{Heterogeneous population}
Consider the case where the population of users traversing the network are divided into $K$ groups, each containing users that share a common prior. In particular, assume that ${c^k \in (0,1]}$ is the fraction of users sharing the prior $q[k] \in \Delta_1^m$ and we have $\sum_{k = 1}^K c^k = 1$. We assume that each group $k \in [K]$ uses the same set of available paths. Note that we considered $K = 1$ in the earlier sections. After a public signal $\zeta^u$ is received, each group $k$ routes its fraction of the flow according to the $\widetilde{q}^{\zeta^u}[k]$-WE, where $\widetilde{q}^{\zeta^u}[k]$ is the posterior formed by group $k$ under a signal $\zeta^u$ and some signalling scheme $\Phi$. The aggregated flow observed by the TIS is 
\begin{equation} \label{eq:aggregate-flow}
	\widetilde{f}^{\zeta^u} := \sum_{k \in [K]} c^k \widetilde{f}^{u,k},
\end{equation}
where $\widetilde{f}^{u,k}$ is a $\widetilde{q}^{\zeta^u}[k]$-WE.
	
First, we note that for the case $K = 2$, where $c^1,c^2$ and $q[1]$ are known, then identification of $q[2]$ can be achieved by following Algorithm~\ref{ag:DQSS}. This is so because for each signal we observe $\widetilde{f}^{\zeta^u}$ while we know $\widetilde{f}^{u,1}$. Thus, following~\eqref{eq:aggregate-flow}, one gets $\widetilde{f}^{u,2} = \frac{\widetilde{f}^{\zeta^u} - c^1 \widetilde{f}^{u,1}}{c^2}$.  Identification of $q[2]$ can then be done using Algorithm~\ref{ag:DQSS} by perceiving the second group as the only one being routed. Next examine the case where more than one prior is unknown. Here, even when the fractions $c^1$ and $c^2$ are known, it is not clear how to design a signalling scheme that can identify both priors. The reason being that now we have an additional $m$ unknowns as compared to the case of single prior, while the amount of information that can be obtained from a signalling scheme does not grow.
	
Finally, consider the case where all priors $\{q[k]\}$ are known, but the fractions $\{c^k\}$ are not. Here, for a given signalling scheme ${\Phi \in \colsto{z}{m}}$, we define the following matrix:
\begin{equation} \label{eq:matrix-of-ind-flows}
	M := \left(\begin{array}{cccc}
		1	&1	&\cdots	&1
		\\
		\widetilde{f}^{1,1}	&\widetilde{f}^{1,2}	&\cdots	&\widetilde{f}^{1,K}
		\\
		\widetilde{f}^{2,1}	&\ddots	&	&\vdots
		\\
		\vdots	&	&	&
		\\
		\widetilde{f}^{z,1}	&\widetilde{f}^{z,2}	&\cdots	&\widetilde{f}^{z,K}
	\end{array}\right)
\end{equation}
and present the following result.
\begin{lemma} \longthmtitle{Identifying population size per prior}
	Let $\PP$, $\Theta$, $\CC$ be given, together with pairs of fractions and priors $\{(c^k,q^k)\}_{k \in [K]}$, $K \in \naturalnumbers$ satisfying $c^k > 0$ for all $k \in [K]$ and $q^k \neq q^\ell$ for all $k \neq \ell$. A signalling scheme ${\Phi \in \colsto{z}{m}}$ allows us to uniquely identify the vector ${c := (c^1,c^2,\cdots,c^K)^\top}$ if and only if
	\begin{equation*}
		\textnormal{rank}(M) = K.
	\end{equation*}
\end{lemma}
\begin{proof}
	We know that $c$ must satisfy $\mathbb{1}^\top c = 1$, since the fractions sum up to the whole of the population. This, together with \eqref{eq:aggregate-flow} and \eqref{eq:matrix-of-ind-flows} implies that $c$ must satisfy
	\begin{equation} \label{eq:M-times-c}
		M c = \left( 1, \enskip \widetilde{f}^{\zeta^1}, \enskip \widetilde{f}^{\zeta^2}, \enskip \cdots, \widetilde{f}^{\zeta^z} \right)^\top. 
	\end{equation}
	When $\textnormal{rank}(M) = K$, that is, $M$ has full column rank, the above equation has a unique solution. If on the other hand $\textnormal{rank}(M) < K$, then the equality~\eqref{eq:M-times-c} still holds. However, in this case there also exists $\widetilde{c} \in \real^K$ such that $M\widetilde{c} = 0$ and $\widetilde{c} \neq 0$. Since $c > 0$, there exists $\epsilon > 0$ such that $c + \epsilon \widetilde{c} \geq 0$. We then have $M(c + \epsilon \widetilde{c}) = Mc$, which implies $c + \epsilon \widetilde{c}$ is in $\Delta_1^K$ and is a solution to \eqref{eq:M-times-c}. In other words, there exist multiple solutions to \eqref{eq:M-times-c} in $\Delta_1^K$.
\end{proof}
In general it is difficult to prescribe guidelines on how to design $\Phi$ in order to ensure that $M$ has full row rank. However, when $z \geq K$ and flows $\{f^{\theta_{s}}\}_{s \in [k]}$ are linearly independent, one can design the signal $\zeta^k$ such that $\widetilde{q}^{\zeta^k}[\ell]$ is arbitrarily close to $q^{\theta_{k}}$. In this way, the induced WE $\widetilde{f}^{u,\ell}$ will get arbitrarily close to $f^{\theta_{k}}$ for all $\ell$. Since flows $\{f^{\theta_{s}}\}_{s \in [K]}$ are linearly independent, this will result in $M$ having full column rank. Also note that when considering $K = 2$, all that is required is that there exist $k, \ell \in [K]$ and a $u \in [z]$ such that $\widetilde{f}^{u,k} \neq \widetilde{f}^{u,\ell}$. 
\subsection{Robustness of signalling schemes in identifying priors}
One of the limitations of our results is that we consider the prior distribution that the population adheres to as fixed. However, we have the following robustness result on $q$-identifying signalling schemes with respect to perturbations in the prior, which shows that for a $q$-identifying signalling scheme $\Phi$ for which the obtained equality constraints are enough to identify $q$, there exists a neighbourhood of $q$ such that for all priors $\widehat{q}$ in this neighbourhood $\Phi$ is $\widehat{q}$-identifying.
\begin{lemma}\longthmtitle{Robustness of $\Phi$ for identifying prior}
	Let $\PP$, $\Theta$, $\CC$, a prior $q$, and a $q$-identifying signalling scheme $\Phi$ be given. In addition, let $\QQ_{\Phi}^=$ be defined as
	\begin{align}
		\QQ_{\Phi}^= \! :=  \! \!  \Set{ \! \! \! \begin{array}{l} \! \! \!  \varphi \! \in \! \real^m \end{array} \! \! \! \! \! | \! \! \! \! \begin{array}{l} \hspace{55 pt} \mathbb{1}^\top\varphi 	= 1,	
		\\
		\sum \limits_{s \in [m]} \! \! \phi^u_s \big(C_p^{\theta_{s}}(\widetilde{f}^{\zeta^u} \! ) \! - \!  C_r^{\theta_{s}}(\widetilde{f}^{\zeta^u} \! )\big)\varphi_s \! = \! 0
		\\
		\forall u \! \in \! [z], \enskip p,r \! \in \! \PP \textnormal{ with } \widetilde{f}^{\zeta^u}_p, \widetilde{f}^{\zeta^u}_r \! > 0.	
		\end{array} \! \! \! \! \!  } \! \! . \label{eq:Q-equality}
	\end{align}
	If $\QQ_{\Phi}^= = \{q\}$, then there exists a $\delta > 0$ such that for all $\widehat{q} \in \Delta_1^m$ with $\norm{q - \widehat{q}} < \delta$ the signalling scheme $\Phi$ is $\widehat{q}$-identifying.
\end{lemma}
\begin{proof}
	First, we note that as a consequence of Corollary~\ref{cor:Q-independent-of-specific-flow}, the set $\QQ_{\Phi}$ is independent of which WE $\widetilde{f}^{\zeta^u} \in \WW^{\zeta^u}$ are observed. Thus, to show that $\Phi$ is $\xi$-identifying for ${\xi \in \Delta_1^m}$, it is enough to show that there exists a set $\{\widetilde{f}^{\zeta^u}\}_{u \in [z]}$ of $\widetilde{\xi}^{\zeta^u}$-WE such that the obtained constraints identify $\xi$. Now consider $q \in \Delta_1^m$ and the signalling scheme $\Phi$ which by assumption is $q$-identifying. Since $\QQ_{\Phi}^= = \{q\}$, it follows from~\eqref{eq:Q-equality} that there exist $m-1$ triplets $\{(p_i,r_i,u_i)\}_{i \in [m-1]}$ with $p_i, r_i \in \PP$ and $u_i \in [m]$ such that $\widetilde{f}^{\zeta^{u_i}}_{p_i}, \widetilde{f}^{\zeta^{u_i}}_{r_i} > 0$ and the system of equations 
	\begin{equation}\label{eq:Q-matrix-equality}
		Q(q) \varphi \! := \! 
		\left( 
			\begin{array}{c} \! \! \! \!
				\mathbb{1}^\top	
				\\
				\big(\alpha(q,p_1,r_1,u_1)\big)^\top
				\\
				\vdots
				\\
				\big(\alpha(q,p_{m-1},r_{m-1},u_{m-1})\big)^\top \! \! \! \!
			\end{array} 
		\right) \varphi \! = \! 
		\left( 
			\begin{array}{c}
				1	\\
				0	\\
				\vdots	\\
				0
			\end{array}
		\right)
	\end{equation}
	has one solution $\varphi = q$. Here $\alpha(q,p_i,r_i,u_i) \in \real^{m}$ is given by
	\begin{equation*}
		\alpha(q,p_i,r_i,u_i) := 
		\left(
			\begin{array}{c}
				\phi^u_1 \big( C_{p_i}^{\theta_{1}}(\widetilde{f}^{\zeta^{u_i}} ) - C_{r_i}^{\theta_{1}}(\widetilde{f}^{\zeta^{u_i}} ) \big)
				\\
				\phi^u_2 \big( C_{p_i}^{\theta_{2}}(\widetilde{f}^{\zeta^{u_i}} ) - C_{r_i}^{\theta_{2}}(\widetilde{f}^{\zeta^{u_i}} ) \big)
				\\
				\vdots
				\\
				\phi^u_m \big( C_{p_i}^{\theta_{m}}(\widetilde{f}^{\zeta^{u_i}} ) - C_{r_i}^{\theta_{m}}(\widetilde{f}^{\zeta^{u_i}} ) \big)
			\end{array}
		\right),
	\end{equation*}
	where the dependence on $q$ is via the dependence of $\widetilde{f}^{\zeta^{u_i}}$ on the posterior $\widetilde{q}^{\zeta^{u_i}}$. Note that from~\eqref{eq:Q-matrix-equality}, the matrix $Q(q)$ has full rank, that is, $\textnormal{rank}\big(Q(q)\big) = m$. From Lemma~\ref{lem:continuity-edge-flow}, since $\widetilde{f}^{\zeta^{u_i}}_{p_i}, \widetilde{f}^{\zeta^{u_i}}_{r_i} > 0$, there exists a $\delta_f > 0 $ such that $\norm{q - \widehat{q}} < \delta_f$ implies that for posteriors $\widehat{q}^{\zeta^{u_i}}$ (based on the prior $\widehat{q}$), there exist $\widehat{q}^{\zeta^{u_i}}$-WE, denoted $\widehat{f}^{\zeta^{u_i}}$, satisfying $\widehat{f}^{\zeta^{u_i}}_{p_i}, \widehat{f}^{\zeta^{u_i}}_{r_i} > 0$. That is, positive flow on $p_i$ and $r_i$ under a WE formed using signal $\zeta^{u_i}$ under scheme $\Phi$ and prior $q$ implies that same paths will have positive flow for some WE under the same signal $\zeta^{u_i}$ and scheme $\Phi$ but induced by a prior $\widehat{q}$ that is close enough to $q$.   This fact along with~\eqref{eq:edge-flows-defined},~\eqref{eq:path-cost-from-edges}, Lemma~\ref{lem:continuity-edge-flow}, and the continuity of functions $\{C_{e_k}(\cdot)\}$, implies that the entries of the matrix $Q(q)$ change continuously with respect to $q$. That is, there exists a $\delta_Q > 0 $ such that $\norm{q - \widehat{q}} < \delta_Q$ implies $\textnormal{rank}\big(Q(\widehat{q})\big) = m$.\footnote{To see this, take a square, non-singular submatrix of $Q(q)$ and note that the determinant depends continuously on the coefficients of $Q(q)$. It follows that for small enough perturbations, the determinant of the submatrix does not becomes $0$, and thus $Q(q)$ retains full column rank.} Thus, the linear system of equations $Q(\widehat{q})\varphi = (1 \enskip 0 \enskip \cdots \enskip 0)^\top$ has a unique solution which is necessarily $\widehat{q}$. That is, we have $\QQ_{\Phi} = \{\widehat{q}\}$. As we mentioned before, even though we use here that for specific $\widehat{q}^{\zeta^{u_i}}$-WE we obtain $\widehat{f}^{\zeta^{u_i}}_{p_i}, \widehat{f}^{\zeta^{u_i}}_{r_i} > 0$, which supplies us with the required equality constraints, the set $Q_\Phi$ is independent of which specific $\widehat{q}^{\zeta^{u_i}}$-WE is observed. The result follows.
\end{proof}
We note that the signalling schemes produced by Algorithm~\ref{ag:DQSS} are of the type considered in the above result. That is, Algorithm~\ref{ag:DQSS} produces signalling schemes for which the resulting equality constraints are enough to identify $q$. What is more, since for these schemes each signal (except the first) supplies one independent equality constraint relating two elements $q_k$ and $q_{\ell}$ of the prior, each signal can be analyzed separately to find the region of priors for which it is guaranteed to still supply an equality constraint. For example, let $\Phi$ be a signalling scheme for a given instance of the game such that under the signal $\zeta^2$ we have $\widetilde{q}_s = 0$ for all $s \geq 3$ and let the remaining dependency of the WE on the posterior $\widetilde{q}$ be given in Figure \ref{fig:2-road-2-state}, where $\widetilde{q}_1 = p_1$. If $\widetilde{q}_1 \in (0.133,0.8)$ the resulting flow gives us an equality constraint on $q$. Additionally, for any perturbation $\widehat{q}$ of $q$, the signalling scheme $\Phi$ will still provide an equality constraint on $\widehat{q}$ as long as the induced change in posterior does not take it outside of the set $(0.133,0.8)$. Furthermore, once $q$ has been identified, the scheme can be modified so as to ensure $\widetilde{q}_1$ is in the center of the interval $(0.133,0.8)$ thereby increasing the robustness of this signal for identification purposes.
\section{Conclusions}\label{sec:conclusions}
In the context of routing games, we have investigated how a TIS can derive information about the prior believes of a population by observing the equilibrium flows induced by different public signals containing information about the state of the network. We have shown that under mild assumptions there always exist signalling schemes that will allow the TIS to fully learn the prior of the population. We have provided an algorithm for updating a given signalling scheme step by step in order to find a scheme sufficient for identifying the held believes. In addition we have shown that a subset of schemes sufficient for identifying the prior are robust in the sense that they can still identify the prior after it has been perturbed by a small amount. We have also briefly investigated the case where the population is divided among several known priors, and given conditions for when the fraction of the population associated to each prior can be identified. We have used examples to illustrate our results.

In the future, we aim to expand our results to more realistic scenarios. Most importantly we hope to further investigate the case where the population is divided among multiple priors. Other possible directions of research are when only noisy observations of the WE are available, when private signalling schemes are used or when the support of the states is not finite. Another interesting but challenging line of research is to investigate optimality of signalling schemes when balancing the objective of gaining information about the prior and minimizing the social cost.
\bibliographystyle{ieeetr}
\bibliography{bibliography.bib}
\end{document}